\documentclass[11pt,a4paper]{article}

\usepackage{epsf,epsfig,amsfonts,amsgen,amsmath,amstext,amsbsy,amsopn,amsthm
}
\usepackage{amsmath,times,mathptmx}
\usepackage{amsfonts,amsthm,amssymb}
\usepackage{amsfonts}
\usepackage{graphics}
\usepackage{latexsym,bm}
\usepackage{amsfonts,amsthm,amssymb,bbding,amsmath,}
\usepackage{indentfirst}
\usepackage{graphicx}
\usepackage{color}
\usepackage[colorlinks=true,anchorcolor=blue,filecolor=blue,linkcolor=blue,urlcolor=blue,citecolor=blue]{hyperref}
\usepackage{float}
\usepackage{tikz}
\newtheorem{thm}{Theorem}

\newtheorem{lemma}{Lemma}
\newtheorem{false statement}{False statement}

\theoremstyle{definition}

\newtheorem{claim}{Claim}
\newtheorem{claimm}{Claim}

\newtheorem{case}{Case}
\newtheorem{casess}{Case}
\newtheorem{casesss}{Case}
\newtheorem{casessss}{Case}

\newtheorem{case1}{Case}

\newtheorem{case3}{Case}
\newtheorem{case4}{Case}

\newtheorem{subcase1}{Subcase}[case]

\begin{document}
\title{The index of unbalanced signed complete graphs whose negative-edge-induced subgraph is ${K}_{2,2}$-minor free
	\footnote{Supported by Natural Science Foundation of Xinjiang Uygur Autonomous Region (No. 2024D01C41), NSFC (No. 12361071).}}
\author{ Mingsong Qin, Dan Li\thanks{Corresponding author. E-mail: ldxjedu@163.com.} \\
	{\footnotesize College of Mathematics and System Science, Xinjiang University, Urumqi 830046, China}}
\date{}

\maketitle {\flushleft\large\bf Abstract:}
Let $\Gamma=(K_n,H^-)$ be a signed complete graph with the negative edges induced subgraph $H$. According to the properties of the negative-edge-induced subgraph, characterizing the extremum problem of the index of the signed complete graph is a concern in signed graphs. A graph $G$ is called $H$-minor free if $G$ has no minor which is isomorphic to $H$. In this paper, we characterize the extremal signed complete graphs that achieve the maximum and the second maximum index when $H$ is a $K_{2,2}$-minor free spanning subgraph of $K_n$.
\vspace{0.1cm}
\begin{flushleft}
	\textbf{Keywords:} Signed complete graph; Tur\'{a}n problem; Adjacency matrix; Spectral radius
\end{flushleft}
\textbf{AMS Classification:} 05C50; 05C35

\section{Introduction}
Let $G$ be a simple graph with the vertex set $V(G)=(v_1,...,v_n)$ and the edge set $E(G)=(e_1,...,e_n)$. The order and size of $G$ are defined as $|V(G)|$ and $|E(G)|$, respectively. A subgraph $H$ of $G$ is called a spanning subgraph if $|V(H)|=|V(G)|$. The degree of a vertex $v_i$ in $G$ is denoted by $d_G(v_i)$ which is the number of edges incident with $v_i$. If $d_G(v_i)=1$, then vertex $v_i$ is called a pendant vertex. We denote the set of all neighbors of $u$ in $G$ by $N_G(u)$. Let $K_n$ be the complete graph of order $n$. As usual, $P_{k}$, $K_{1,k}$ and $C_k$ denote the path of order $k$, the star of order $k+1$ and the cycle of order $k$, respectively. An underlying graph $G$ with a sign transform $\sigma:E(G)\to\{-1,+1\}$  make up a signed graph $\Gamma=(G,\sigma)$. In sigend graph, edge signs are usually interpreted as $\pm1$. An edge $e$ is positive (resp.negative) if $\sigma(e)=+1$ (resp. $\sigma(e)=-1$). A cycle in $\Gamma$ is said to be positive if it contains an even number of negative edges, otherwise the cycle is negative. $\Gamma=(G,\sigma)$ is balanced if there are no negative cycles, otherwise it is unbalanced. For more details about the notion of signed graphs, we refer to \cite{1}. Signed graph was first introduced in works of Harary \cite{8} and Cartwright and Harary \cite{5}, and the matroids of graphs were extended to matroids of signed graphs by Zaslavsky \cite{19}. Chaiken \cite{6} and Zaslavsky \cite{19} obtaind the Matrix-Tree Theorem for signed graph independently. The theory of signed graphs is a special case of that of gain graphs and of biased graphs \cite{20}. The adjacency matrix of $\Gamma$ is defined as $A(\Gamma)=(a_{ij}^\sigma)$. Then $a_{uv}=\sigma(uv)$ if $u\sim v $, otherwise, $a_{uv}=0$. The eigenvalues of $\Gamma$ are written by $\lambda_1(\Gamma)\geq\lambda_{2}(\Gamma)\geq\cdots\geq\lambda_n(\Gamma)$  with decreasing  order which are the eigenvalues of $A(\Gamma)$ and the index of $\Gamma$ is the $\lambda_1(\Gamma)$. 

Extremal graph theory deals with the problem of determining extremal values and extremal graphs for a given graph invariant in a given set of graphs. It is worth noting that  Brunetti and Stani\'{c} \cite{4} studied the extremal spectral radius among all unbalanced connected signed graphs.  Let $\Gamma=(K_n,H^-)$ be a signed complete graph with the negative edge-induced subgraph $H$. In 2023, Li, Lin and Meng \cite{12} gave the maximum $\lambda_1(A(\Sigma))$  among graphs of type $\Sigma=(K_n,T^-)$, where $T$ is a spanning tree of $K_n$. More results on the index of signed graphs can be found in \cite{3,9}. Let $H$ be a simple graph. A graph $G$ is $H$-free if no subgraph of $G$ is isomorphic to $H$. The classical spectral Tur\'{a}n problem is to determine the maximum spectral radius of a $H$-free graph, which is known as the spectral Tur\'{a}n number of $H$. This problem was originally proposed by Nikiforov \cite{13}. Tur\'{a}n \cite{14} raised and solved the extremal problem for ${K}_{r}$-free graphs with $r\geq 3$. More about spectral Tur\'{a}n problem for unsigned graphs see \cite{t2,t6}. A graph $H$ is a minor of a graph $G$ if $H$ can be obtained from $G$ by deleting edges, deleting vertices and contracting edges. A graph $G$ is called $H$-minor free if $G$ has no minor which is isomorphic to $H$. In particular, Wagner \cite{m3}  proved that a graph is planar if and only if it is $K_5$-minor and $K_{3,3}$-minor free. Similar to Wagner’s result, Ding and Oporowski \cite{m4} showed that a graph is outerplanar if and only if it is $K_4$-minor and $K_{2,3}$-minor free. In spectral extremal graph theory, it is interesting to determine the maximum spectral radius of $H$-minor free graphs for some given $H$. For example, In 2019, Tait \cite{100} determined the maximum spectral radius of $K_r$-minor free graph of order $n$, also presented a conjecture on the extremal graph that achieves the maximum spectral radius among $K_{s,t}$-minor free graphs. In 2022, Zhai and Lin \cite{t5} proved the conjecture for sufficiently large order.

In this paper, we  focus on the spectral Tur\'{a}n problem in signed graphs. Let $K_r^-$ and $C_r^-$ be the sets of all unbalanced signed graphs with underlying graphs $K_r$ and $C_r$, respectively. Given a set $F$ of signed graphs, if a signed graph $\Gamma$ contains no signed subgraph isomorphic to any one in $F$, then $\Gamma$ is called $F$-free. Chen and Yuan \cite{7} gave the spectral Tur\'{a}n number of ${K}_{4}^{-}$. For $r \geq 5$, the $K_r^-$-free unbalanced signed graphs of fixed order $n$ with maximum index (resp. spectral radius $r\leq \frac n2$) have been determined in \cite{k1,k2}. In 2022, Wang, Hou and Li \cite{15} determined the spectral Tur\'{a}n number of ${C}_{3}^{-}$. The ${C}_{4}^{-}$-free unbalanced signed graphs of fixed order with maximum index have been determined by Wang and Lin \cite{16}. Moreover, the ${C}_{2k+1}^{-}$-free unbalanced signed graphs of fixed order $n$ with maximum index have been determined in \cite{18}, where $3 \leq k \leq n/10-1$. 
\begin{figure}
	\centering
	\ifpdf
	\setlength{\unitlength}{1bp}%
	\begin{picture}(242.09, 114.69)(0,0)
		\put(0,0){\includegraphics{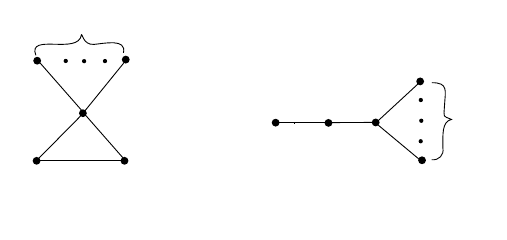}}
		\put(27.77,100.14){\fontsize{11.38}{13.66}\selectfont \textit{$n-3$}}
		\put(35.99,21.14){\fontsize{11.38}{13.66}\selectfont $U_1$}
		\put(151.41,21.67){\fontsize{11.38}{13.66}\selectfont $D_{1,n-3}$}
		\put(44.08,58.34){\fontsize{11.38}{13.66}\selectfont \textit{$v_1$}}
		\put(5.67,34.82){\fontsize{11.38}{13.66}\selectfont \textit{$v_2$}}
		\put(63.53,35.62){\fontsize{11.38}{13.66}\selectfont \textit{$v_3$}}
		\put(221.30,53.97){\fontsize{11.38}{13.66}\selectfont \textit{$n-3$}}
		\put(29.40,0.1){\fontsize{11.38}{13.66}\selectfont $\bf {Fig.   1.}\label{Fig1}$ The graphs $U_1$ and $D_{1,n-3}$.}
	\end{picture}%
	\else
	\setlength{\unitlength}{1bp}%
	\begin{picture}(242.09, 114.69)(0,0)
		\put(0,0){\includegraphics{f1}}
		\put(31.77,100.14){\fontsize{11.38}{13.66}\selectfont \textit{$n-3$}}
		\put(35.99,28.14){\fontsize{11.38}{13.66}\selectfont $U_1$}
		\put(151.41,29.67){\fontsize{11.38}{13.66}\selectfont D\symbol{95}\{1,n-3\}}
		\put(44.08,58.34){\fontsize{11.38}{13.66}\selectfont \textit{$v_1$}}
		\put(5.67,34.82){\fontsize{11.38}{13.66}\selectfont \textit{$v_2$}}
		\put(63.53,35.62){\fontsize{11.38}{13.66}\selectfont \textit{$v_3$}}
		\put(221.30,53.97){\fontsize{11.38}{13.66}\selectfont \textit{$n-3$}}
		\put(22.64,8.12){\fontsize{11.38}{13.66}\selectfont $\bf {Fig.1.}$ The graph $U_1$ and $D(1,n-3)$.}
	\end{picture}%
	\fi
\end{figure}
  Motivated by these works, we are interested in what the extremal graph with the maximum index in the class of signed graphs whose negative edge-induced subgraphs are $H$-minor free. In this paper, we consider this problem  in the unbalanced signed complete graphs whose negative-edge-induced subgraph is ${K}_{2,2}$-minor free. Let $U_1$ $(Fig. \ref{Fig1})$ denote a triangle with all remaining vertices being pendant at the same vertex of the triangle, and $D_{1,n-3}$ $(Fig. \ref{Fig1})$ be obtained by adding $n-3$ pendent vertices to one end vertex of $P_3$. We know that $H\cong U_1$ or $H \cong P_4$ or $H \cong 2P_2$ for signed complete graph $\Gamma=(K_4,H^-)$. Since $(K_4,U_1^-)$ is switching isomorphic $(K_4,P_4^-)$,  $\lambda_1(A((K_4,U_1^-)))=\lambda_1(A((K_4,P_4^-)))> \lambda_1(A((K_4,2P_2^-)))$ by calculation. Thus, we only need to consider $n\geq 5$, and the main result of this paper is as follows.

\begin{thm}\label{thm1}
Let $H$ be a $K_{2,2}$-minor free spanning subgraph of $K_n$ for $n\geq 5$. If $\Gamma$ is not switching isomorphic  $(K_n,U_1^-)$ and $(K_n,D_{1,n-3}^-)$,  then  $\lambda_1(A((K_n,U_1^-)))>\lambda_1(A((K_n,D_{1,n-3}^-)))> \lambda_1(A(\Gamma))$.
\end{thm}

The remainder of this paper is organized as follows. In Section \ref{se2}, we will investigate lemmas that will be useful in proving Theorem \ref{thm1}. Furthermore, Section \ref{se3} will review some established results, investigate lemmas pertinent to the proof of Theorem \ref{thm1}, and present its proof.

\section{Preliminaries}\label{se2}
Let $M$ be a real symmetric matrix with block from $M= [M_{ij}]$, and $q_{ij}$ be the average row sum of $M_{ij}$ and the matrix $Q=(q_{ij})$ be the quotient matrix of $M$. Furthermore, $Q$ is referred to as an equitable quotient matrix if every block $M_{ij}$ has a constant row sum. Let Spec$(Q)=\{\lambda_1^{[t_1]},...,\lambda_k^{[t_k]}\}$ be the spectrum of $Q$ with eigenvalue 
$\lambda_i$ with multiplicities $t_i$ for $1 \leq i \leq k$.
\begin{lemma}\label{l}\cite{21}.
	There are two kinds of eigenvalues of the real symmetric matrix $M$.

	(i) The eigenvalues match the eigenvalues of $Q$.
	
	(ii) The eigenvalues of $M$ not in Spec$(Q)$ are unchanged when $\alpha $J is add to block $M_{ij}$ for every $1\leq i,j \leq m$, where  $\alpha$ is any constant. Moreover, $\lambda_1(M)=\lambda_1(Q)$ when $M$ is irreducible and nonnegative.
	
\end{lemma}
Let $\Gamma=(G,\sigma)$ be a signed graph and  $\varphi_{\Gamma}(\lambda)=det(\lambda I-A(\Gamma))$ be the characteristic polynomial of $\Gamma$. The matrix $J_{r\times s}$ is all-one matrix of size $r \times s$, and when $r=s$ it is denoted by $J_r$. Also, we use $j_k=(1,\ldots,1)^T\in R^k$.
\begin{lemma}\label{z}
	Let $\Gamma=(K_n,D_{1,n-3}^-)$ and $\Gamma^{\prime}=(K_n,U_1^-)$.  Then $\lambda_1(A(\Gamma))<\lambda_1(A(\Gamma^\prime))$ for $n \geq 5$.
\end{lemma}
\begin{proof}
By \cite{12}, the characteristic polynomial of $\Gamma$ is
\begin{center}
$\varphi_{\Gamma}(\lambda)=(\lambda+1)^{n-3}(\lambda^3+(3-n)\lambda^2+(3-2n)\lambda+7n-23)$.  
\end{center}
For  $\Gamma^{\prime}=(K_n,U_1^-)$, the vertex partition is $V_1=\{v_1\}$, $V_2=\{v_2\}$, $V_3=\{v_3\}$ and $V_4=\{v_4, ... ,v_n\}$. Then $A(\Gamma^{\prime})$ is given by the following
\begin{center}
$A(\Gamma^{\prime})=\begin{bmatrix}0&-1&-1&-j_{n-3}^T\\-1&0&-1&j_{n-3}^T\\-1&-1&0&j_{n-3}^T\\-j_{n-3}&j_{n-3}&j_{n-3}&\left(J-I\right)_{n-3}\end{bmatrix}$.
\end{center}
Hence,
\begin{center}
$\lambda I_n-A(\Gamma^{\prime})=\begin{bmatrix}\lambda&1&1&j_{n-3}^T\\1&\lambda&1&-j_{n-3}^T\\1&1&\lambda&-j_{n-3}^T\\j_{n-3}&-j_{n-3}&-j_{n-3}&\left((\lambda+1)I-J\right)_{n-3}\end{bmatrix}$.
\end{center}
Now, we apply finitely many elementary row and column operations on the matrix $\lambda I_n-A(\Gamma^{\prime})$. First, subtracting the $4$-th row from all the lower rows and adding the $i$-th column to the $4$-th column for $i=n,...,5$. This leads to the following matrix:
\begin{center}
$\begin{gathered}\lambda I_n-A(\Gamma^{\prime})=\begin{bmatrix}\lambda&1&1&n-3&*\\1&\lambda&1&3-n&*\\1&1&\lambda&3-n&*\\1&-1&-1&\lambda-n+4&*\\\bf{0}&\bf{0}&\bf{0}&\bf{0}&\left(\lambda+1\right)I_{n-4}\end{bmatrix}\end{gathered}$.
\end{center} 
Therefore, $\varphi_{\Gamma^\prime}(\lambda)=(\lambda+1)^{n-3}(\lambda-1)(\lambda^2+(4-n)\lambda+7-3n)$. Let $f_{1}(\lambda)=(\lambda^3+(3-n)\lambda^2+(3-2n)\lambda+7n-23)$ and $f_{2}(\lambda)=(\lambda-1)(\lambda^2+(4-n)\lambda+7-3n)$.
Note that $f_{2}(\lambda)-f_{1}(\lambda)=16-4n<0$ for $n \geq 5$, then $\varphi(\Gamma^{\prime},\lambda_1(A(\Gamma)))<0$. This implies that $\lambda_1(A(\Gamma))<\lambda_1(A(\Gamma^\prime))$.
\end{proof}
\begin{figure}[H]
	\centering
	\ifpdf
	\setlength{\unitlength}{1bp}%
	\begin{picture}(296.78, 134.77)(0,0)
		\put(0,0){\includegraphics{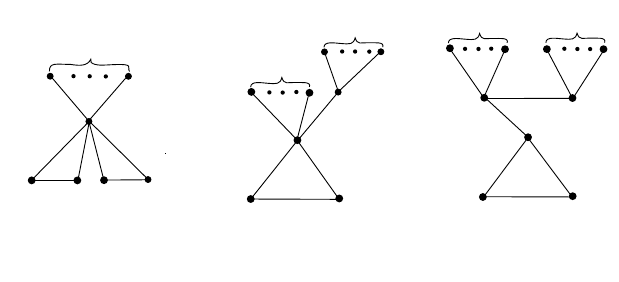}}
		\put(106.90,34.80){\fontsize{11.38}{13.66}\selectfont $v_2$}
		\put(165.50,34.80){\fontsize{11.38}{13.66}\selectfont $v_3$}
		\put(145.86,65.90){\fontsize{11.38}{13.66}\selectfont \textit{$v_1$}}
		\put(278.71,38.16){\fontsize{11.38}{13.66}\selectfont \textit{$v_3$}}
		\put(275.64,82.24){\fontsize{11.38}{13.66}\selectfont \textit{$v_5$}}
		\put(168.24,118.50){\fontsize{11.38}{13.66}\selectfont \textit{$t$}}
		\put(133.99,98.85){\fontsize{11.38}{13.66}\selectfont \textit{$s$}}
		\put(228.09,120.21){\fontsize{11.38}{13.66}\selectfont \textit{$s$}}
		\put(275.71,120.00){\fontsize{11.38}{13.66}\selectfont \textit{$t$}}
		\put(238.00,28.32){\fontsize{11.38}{13.66}\selectfont \textit{$H_2(s,t)$}}
		\put(58.87,8.12){\fontsize{11.38}{13.66}\selectfont $\bf{Fig.2.}\label{Fig2}$ The graphs $U_1$, $H_1(s,t)$ and $H_2(s,t)$.}
		\put(220.57,37.39){\fontsize{11.38}{13.66}\selectfont \textit{$v_2$}}
		\put(257.18,67.34){\fontsize{11.38}{13.66}\selectfont \textit{$v_1$}}
		\put(219.89,82.34){\fontsize{11.38}{13.66}\selectfont \textit{$v_4$}}
		\put(125.33,27.86){\fontsize{11.38}{13.66}\selectfont \textit{$H_1(s,t)$}}
		\put(159.18,81.53){\fontsize{11.38}{13.66}\selectfont \textit{$v_4$}}
		\put(31.10,107.43){\fontsize{11.38}{13.66}\selectfont \textit{$n-5$}}
		\put(36.86,27.61){\fontsize{11.38}{13.66}\selectfont $U_2$}
		\put(46.15,74.59){\fontsize{11.38}{13.66}\selectfont \textit{$v_1$}}
		\put(28.70,39.75){\fontsize{11.38}{13.66}\selectfont \textit{$v_3$}}
		\put(65.47,39.75){\fontsize{11.38}{13.66}\selectfont \textit{$v_5$}}
		\put(5.67,39.75){\fontsize{11.38}{13.66}\selectfont \textit{$v_2$}}
		\put(44.56,39.75){\fontsize{11.38}{13.66}\selectfont \textit{$v_4$}}
	\end{picture}%
	\else
	\setlength{\unitlength}{1bp}%
	\begin{picture}(296.78, 134.77)(0,0)
		\put(0,0){\includegraphics{f2}}
		\put(106.90,34.80){\fontsize{11.38}{13.66}\selectfont v2}
		\put(165.50,37.78){\fontsize{11.38}{13.66}\selectfont v3}
		\put(145.86,65.90){\fontsize{11.38}{13.66}\selectfont \textit{v1}}
		\put(278.71,38.16){\fontsize{11.38}{13.66}\selectfont \textit{v3}}
		\put(275.64,82.24){\fontsize{11.38}{13.66}\selectfont \textit{v5}}
		\put(168.24,118.50){\fontsize{11.38}{13.66}\selectfont \textit{t}}
		\put(133.99,98.85){\fontsize{11.38}{13.66}\selectfont \textit{s}}
		\put(228.09,120.21){\fontsize{11.38}{13.66}\selectfont \textit{s}}
		\put(275.71,120.00){\fontsize{11.38}{13.66}\selectfont \textit{t}}
		\put(247.00,28.32){\fontsize{11.38}{13.66}\selectfont \textit{H2}}
		\put(84.87,8.12){\fontsize{11.38}{13.66}\selectfont Fig.2. The graph H1 and H2.}
		\put(220.57,37.39){\fontsize{11.38}{13.66}\selectfont \textit{v2}}
		\put(257.18,67.34){\fontsize{11.38}{13.66}\selectfont \textit{v1}}
		\put(219.89,82.34){\fontsize{11.38}{13.66}\selectfont \textit{v4}}
		\put(133.33,27.86){\fontsize{11.38}{13.66}\selectfont \textit{H1}}
		\put(159.18,81.53){\fontsize{11.38}{13.66}\selectfont \textit{v4}}
		\put(36.10,107.43){\fontsize{11.38}{13.66}\selectfont \textit{n-5}}
		\put(32.80,27.61){\fontsize{11.38}{13.66}\selectfont U2}
		\put(46.15,74.59){\fontsize{11.38}{13.66}\selectfont \textit{v1}}
		\put(28.70,39.27){\fontsize{11.38}{13.66}\selectfont \textit{v3}}
		\put(65.47,39.05){\fontsize{11.38}{13.66}\selectfont \textit{v5}}
		\put(5.67,39.75){\fontsize{11.38}{13.66}\selectfont \textit{v2}}
		\put(44.56,39.85){\fontsize{11.38}{13.66}\selectfont \textit{v4}}
	\end{picture}%
	\fi
\end{figure}
\begin{lemma}\label{Lem3}
Let $U_2$ be the graph depicted in Fig \ref{Fig2}. If $n\geq 9$, then
\begin{center}
$\lambda_1(A((K_n,D_{1,n-3}^-)))>\lambda_1(A((K_n,U_2^-)))$.
\end{center}
\end{lemma}
\begin{proof}
Notice that the characteristic polynomial of $A((K_n,D_{1,n-3}^-))$ is $(\lambda+1)^{n-3}(\lambda^3+(3-n)\lambda^2+(3-2n)\lambda+7n-23)$ by \cite{12}. For  $(K_n,U_2^-)$, the vertex partition is $V_1=\{v_1\}$, $V_2=\{v_2\}$, $V_3=\{v_3\}$, $V_4=\{v_4\}$, $V_5=\{v_5\}$ and $V_6=\{v_6,...,v_n\}$. Then $A((K_n,U_2^-))$ is given by the following
\begin{center}
	$\begin{gathered}A((K_n,U_2^-))=\begin{bmatrix}0&-1&-1&-1&-1&-j_{n-5}^T\\-1&0&-1&1&1&j_{n-5}^T\\-1&-1&0&1&1&j_{n-5}^T\\-1&1&1&0&-1&j_{n-5}^T\\-1&1&1&-1&0&j_{n-5}^T\\-j_{n-5}&j_{n-5}&j_{n-5}&j_{n-5}&j_{n-5}&\left(J-I\right)_{n-5}\end{bmatrix}\end{gathered}$.
\end{center}
Thus,
\begin{center}
	$\lambda I_n-A((K_n,U_2^-))=\begin{bmatrix}\lambda&1&1&1&1&j_{n-5}^T\\1&\lambda&1&-1&-1&-j_{n-5}^T\\1&1&\lambda&-1&-1&-j_{n-5}^T\\1&-1&-1&\lambda&1&-j_{n-5}^T\\1&-1&-1&1&\lambda&-j_{n-5}^T\\j_{n-5}&-j_{n-5}&-j_{n-5}&-j_{n-5}&-j_{n-5}&\left((\lambda+1)I-J\right)_{n-5}\end{bmatrix}$.
\end{center}
Now, we apply finitely many elementary row and column operations on the matrix $\lambda I_n-A((K_n,U_2^-$ $))$. First, subtracting the $6$-th row from all the lower rows and adding the $i$-th column to the $6$-th column, for $i=n,...,7$. This leads to the following matrix:
\begin{center}
	$\begin{gathered}\lambda I_n-A((K_n,U_2^-))=\begin{bmatrix}\lambda&1&1&1&1&n-5&*\\1&\lambda&1&-1&-1&5-n&*\\1&1&\lambda&-1&-1&5-n&*\\1&-1&-1&\lambda&1&5-n&*\\1&-1&-1&1&\lambda&5-n&*\\1&-1&-1&-1&-1&\lambda+6-n&*\\\bf{0}&\bf{0}&\bf{0}&\bf{0}&\bf{0}&\bf{0}&(\lambda+1)I_{n-6}\end{bmatrix}\end{gathered}$.
\end{center}
Hence, $\varphi((K_n,U_2^-),\lambda)=(\lambda+1)^{n-5}(\lambda-1)^2(\lambda+3)(\lambda^2+(4-n)\lambda+11-3n)$. Let $f_1(\lambda)=(\lambda-1)^2(\lambda+3)(\lambda^2+(4-n)\lambda+11-3n)$ and $f_2(\lambda)=(\lambda+1)^2(\lambda^3+(3-n)\lambda^2+(3-2n)\lambda+7n-23)$. Note that
\begin{center}
	$h(\lambda)=f_1(\lambda)-f_2(\lambda)=8\lambda^2-16n+56$.
\end{center}
The maximal solution of $h(\lambda)=0$ is $\sqrt{2n-7}$ for $n \geq 9$. And
\begin{center}
	$f_1(\frac{n}{2}-1)=(\frac{n-4}{2})^2(\frac{n+4}{2})(-\frac{1}{4}n^2-n+8)$.
\end{center}
Now, let $g_1(n)=-\frac14n^2-n+8$. Notice that $g_1'(n)=-\frac{1}{2}n-1<0$ for $n \geq 9$. Hence, $g_1(n)\leq g_1(9)=-\frac{85}{4}<0$ and $f_1(\frac{n}{2}-1)<0$ for $n\geq 9$. Thus, $\lambda_1(A(K_n,U_2^-))>\frac n2-1>\sqrt{2n-7}$ for $n\geq 9$. Let $\lambda_1=\lambda_1(A((K_n,U_2^-)))$, then
\begin{center}
	$-f_2(\lambda_1)=f_1(\lambda_1)-f_2(\lambda_1)>0$.
\end{center}
This indicates that $\lambda_1(A((K_n,D_{1,n-3}^-)))>\lambda_1(A((K_n,U_2^-)))$.
\end{proof}
Let the matrix $Q$ be the quotient matrix of a real symmetric graph matrix $M$ and $P_{Q}(\lambda)=det(\lambda I-Q)$ denote the characteristic polynomial of $Q$.
\begin{lemma}\label{a}
	Let $s$ and $t$ be non-negative integers such that $s+t=n-4 \geq 5$. If $H_1(s,t)\ncong U_1$, where $H_1(s,t)$ be the graph depicted in $Fig. \ref{Fig2}$, then 
	\begin{center}
		$\lambda_1(A((K_n,D_{1,n-3}^-)))>\lambda_1(A((K_n,H_1(s,t)^-)))$.
	\end{center}
\end{lemma}
\begin{proof}
Assume that $s, t \geq 1$. We give the $A((K_n,H_1(s,t)^-))$ and its corresponding quotient matrix $Q_3(s,t)$ by the vertex partition $V_1=\{v_1\}$, $V_2=\{v_2\}$, $V_3=\{v_3\}$, $V_4=\{v_4\}$, $V_5=N_{H_1(s,t)}(v_1)\backslash\{v_2,v_3,$ $v_4\}$ and $V_6=N_{H_1(s,t)}(v_4)\backslash\{v_1\}$ as follows
\begin{center}
$\begin{gathered}A((K_n,H_1(s,t)^-))=\begin{bmatrix}0&-1&-1&-1&-j_s^T&j_t^T\\-1&0&-1&1&j_s^T&j_t^T\\-1&-1&0&1&j_s^T&j_t^T\\-1&1&1&0&j_s^T&-j_t^T\\-j_s&j_s&j_s&j_s&\left(J-I\right)_s&J_{s\times t}\\j_t&j_t&j_t&-j_t&J_{t\times s}&\left(J-I\right)_t\end{bmatrix}\end{gathered}$,
\end{center}
and
\begin{center}
 $Q_3(s,t)=\begin{bmatrix}0&-1&-1&-1&-s&t\\-1&0&-1&1&s&t\\-1&-1&0&1&s&t\\-1&1&1&0&s&-t\\-1&1&1&1&s-1&t\\1&1&1&-1&s&t-1\end{bmatrix}$.
\end{center}
Note that the characteristic polynomial of $Q_3(s,t)$ is
\begin{align*}
P_{Q_3(s,t)}(\lambda)&=(\lambda-1)(\lambda+1)(\lambda^4-(n-6)\lambda^3-(5n-16)\lambda^2+(8st-7s+9t-10)\lambda\\&+24st-3s+13t-5).
\end{align*}
Adding $\alpha $J to the blocks of $A((K_n,H_1(s,t)^-))$, where $\alpha $ is constant. Then $A((K_n,H_1(s,t)^-))$ will be 
\begin{center}
$A_3=\begin{bmatrix}0&0&0&0&\bf{0}^T&\bf{0}^T\\0&0&0&0&\bf{0}^T&\bf{0}^T\\0&0&0&0&\bf{0}^T&\bf{0}^T\\0&0&0&0&\bf{0}^T&\bf{0}^T\\\bf{0}&\bf{0}&\bf{0}&\bf{0}&-I_s&\bf{0}\\\bf{0}&\bf{0}&\bf{0}&\bf{0}&\bf{0}&-I_t\end{bmatrix}$.
\end{center}
Since $\lambda_1(Q_3(s,t))>0$ and Spec$(A_3)$$=\begin{Bmatrix}-1^{[n-4]},0^{[4]}\end{Bmatrix}$,  $\lambda_1(A((K_n,H_1(s,t)^-)))=\lambda_1(Q_3(s,t))$.\\
Notice that 
\begin{center}
$P_{Q_3(s,t)}\left(\lambda\right)-P_{Q_3(s+1,t-1)}(\lambda)=8\left(s-t+3\right)\lambda+24\left(s-t\right)+40$.
\end{center}
Next, we will discuss in two cases.

\begin{case}\label{case1}
 $s \geq t$. 
\end{case}
Note that
$P_{Q_3(s,t)}(\lambda)-P_{Q_3(s+1,t-1)}(\lambda)>0$ for $\lambda >0$. Then $P_{Q_3(s+1,t-1)}(\lambda_1(Q_3(s,t)))<0$, it means that  $\lambda_1(Q_3(s,t))<\lambda_1(Q_3(s+1,t-1))$. Thus, 
$\lambda_1(A((K_n,H_1(s,t)^-)))<\lambda_1(A((K_n,H_1(s+1,t-1)^-)))$. By repeatedly using this operation, we can obtain that
\begin{center}
$\lambda_1(A((K_n,H_1(s,t)^-))) \leq \lambda_1(A((K_n,H_1(n-5,1)^-)))$,
\end{center}
the equality holds if and only if $s=n-5$, $t=1$.

Next, we will show that $\lambda_1(A((K_n,U_2^-)))>\lambda_1(A((K_n,H_1(n-5,1)^-)))$.
Notice that $P_{Q_3(n-5,1)}$ $(\lambda)=(\lambda-1)(\lambda+1)(\lambda^4-(n-6)\lambda^3-(5n-16)\lambda^2+(n-6)\lambda+21n-97)$ and $\lambda_1(A((K_n,H_1(n-5,1)^-)))$ is the largest root of $P_{Q_3(n-5,1)}(\lambda)=0$. Clearly, $\varphi((K_n,U_2^-),\lambda)=(\lambda+1)^{n-5}(\lambda-1)^2(\lambda+3)(\lambda^2+(4-n)\lambda+11-3n)$. Let $f_3(\lambda)=(\lambda-1)^2(\lambda+1)(\lambda+3)(\lambda^2+(4-n)\lambda+11-3n)$. Hence,
\begin{center}
	$h(\lambda)=P_{Q_3(n-5,1)}(\lambda)-f_3(\lambda)=4(\lambda+1)(\lambda-1)((n-4)\lambda+3n-16)$.
\end{center}
The maximal solution of $h(\lambda)=0$ is $1$ for $n \geq 9$. And
\begin{center}
	$P_{Q_3(n-5,1)}(\frac{n}{2}-1)=(\frac{n-4}{2})(\frac{n}{2})(-\frac{1}{16}n^4-\frac{1}{4}n^3+5n^2+4n-80)$.
\end{center}
Set $f(n)=-\frac1{16}n^4-\frac14n^3+5n^2+4n-80$. It is evidently that $f'(n)=-\frac14n^3-\frac34n^2+10n+4$, $f^{\prime\prime}(n)=-\frac34n^2-\frac32n+10$ and $f^{\prime\prime\prime}(n)=-\frac32n-\frac32$. We observe that $f^{\prime\prime\prime}(n)=-\frac32n-\frac32<0$ for $n \geq 9$. Hence, $f''(n)\le f''(9)=-\frac{257}{4}<0$, and thus $f'(n)\leq f'(9)=-149<0$. This means that $f(n)$ is a monotone decreasing for $n \geq 9$ such that $f(n)\leq f(9)=-\frac{3701}{16}<0$, and then $P_{Q_3(n-5,1)}(\frac n2-1)<0$. 
Thus, $\lambda_{1}(A((K_{n},H_{1}(n-5,1)^{-})))>\frac{n}{2}-1>1$ for $n \geq 9$. Let $\lambda_{1}=\lambda_{1}(A((K_{n},H_{1}(n-5,1)^{-})))$, then
\begin{center}
	$-f_3(\lambda_1)=P_{Q_3(n-5,1)}(\lambda_1)-f_3(\lambda_1)>0$.
\end{center}
This indicates that $\lambda_1(A((K_n,U_2^-)))>\lambda_1(A((K_n,H_1(n-5,1)^-)))$. Thus, $\lambda_1(A((K_n,D_{1,n-3}^-)))>\lambda_1(A((K_n,H_1(s,t)^-)))$ by Lemma \ref {Lem3}.

\begin{case}
 $s<t$. 
\end{case}

\noindent{\bf{${\mbox{subcase 2.1. }}$}}
 $t-s=1$, i.e., $H_1(s,t)\cong H_1(\frac{n-5}{2},\frac{n-3}{2})$. Note that $P_{Q_3(s,t)}(\lambda)-P_{Q_3(s+1,t-1)}(\lambda)=8(s-t+3)\lambda+24(s-t)+40>0$ for $\lambda >0$. Let $\lambda_1=\lambda_1(A((K_n,H_1(s,t)^-)))$, then $P_{Q_3(s+1,t-1)}$ $(\lambda_1)<0$. Thus, $\lambda_1(A((K_n,H_1(\frac{n-5}{2},\frac{n-3}{2})^-)))<\lambda_1(A((K_n,H_1(\frac{n-3}{2},\frac{n-5}{2})^-)))$. Note that $\lambda_1(A((K_n,H_1(\frac{n-3}{2}$ $,\frac{n-5}{2})^-)))$ $<\lambda_1(A((K_n,D_{1,n-3}^-)))$ by Case \ref{case1}. Hence, $\lambda_1(A((K_n,H_1(\frac{n-5}{2},\frac{n-3}{2})^-)))<\lambda_1(A((K_n,$ $D_{1,n-3}^-)))$.

\noindent{\bf{${\mbox{subcase 2.2. }}$}}
 $t-s=2$, i.e., $H_1(s,t)\cong H_1(\frac{n-6}{2},\frac{n-2}{2})$. Note that $P_{Q_{3(s,t)}}\left(\lambda\right)-P_{Q_{3(s+1,t-1)}}(\lambda)=8(s-t+3)$ $\lambda+24(s-t)+40>0$ for $\lambda >1$. Let $\lambda_1=\lambda_1(A((K_n,H_1(s,t)^-)))$, then $P_{Q_3(s+1,t-1)}(\lambda_1)<0$. Thus, $\lambda_1(A((K_n,H_1(\frac{n-6}{2},\frac{n-2}{2})^-)))<\lambda_1(A((K_n,H_1(\frac{n-4}{2},\frac{n-4}{2})^-)))$. Since $\lambda_1(A((K_n,H_1(\frac{n-4}{2},\frac{n-4}{2})^-$ $)))$ $<\lambda_1(A((K_n,D_{1,n-3}^-)))$ by Case \ref{case1}, $\lambda_1(A((K_n,H_1(\frac{n-6}{2},\frac{n-2}{2})^-)))<\lambda_1(A((K_n,D_{1,n-3}^-)))$.

\noindent{\bf{${\mbox{subcase 2.3. }}$}}
 $t-s\geq 3$. Similarly, notice that $P_{Q_3(s,t)}(\lambda)-P_{Q_3(s-1,t+1)}(\lambda)=(8t-8s-8)\lambda+24\left(t-s\right)+8>0$ for $\lambda >0$ and $P_{Q_3(s-1,t+1)}(\lambda_1(Q_3(s,t)))<0$. Hence,  $\lambda_1(Q_3(s,t))<\lambda_1(Q_3(s-1,t+1))$, and then
$\lambda_1(A((K_n,H_1(s,t)^-)))<\lambda_1(A((K_n,H_1(s-1,t+1)^-)))$.  By repeatedly using this operation, we can get that
\begin{center}
	$\lambda_1(A((K_n,H_1(s,t)^-))) \leq \lambda_1(A((K_n,H_1(1, n-5)^-)))$,
\end{center}
the equality holds if and only if $s=1$, $t=n-5$.

Next, we will assert that $\lambda_1(A((K_n,H_1(1,n-5)^-)))<\lambda_1(A((K_n,H_1(0,n-4)^-)))$.
Note that
\begin{align*}
 P_{Q_3(1,n-5)}(\lambda)&=(\lambda-1)(\lambda+1)(\lambda^4-(n-6)\lambda^3-(5n-16)\lambda^2+(17n-102)\lambda\\&+37n-193). 
\end{align*}
By appropriatly marking the vertices of $(K_n,H_1(0,n-4)^-)$, one can have
\begin{center}
$A((K_n,H_1(0,n-4)^-))=\begin{bmatrix}0&-1&-1&-1&j_{n-4}^T\\-1&0&-1&1&j_{n-4}^T\\-1&-1&0&1&j_{n-4}^T\\-1&1&1&0&-j_{n-4}^T\\j_{n-4}&j_{n-4}&j_{n-4}&-j_{n-4}&\left(J-I\right)_{n-4}\end{bmatrix}$.
\end{center}
Then
\begin{center}
$\lambda I_n-A((K_n,H_1(0,n-4)^-))=\begin{bmatrix}\lambda&1&1&1&-j_{n-4}^T\\1&\lambda&1&-1&-j_{n-4}^T\\1&1&\lambda&-1&-j_{n-4}^T\\1&-1&-1&\lambda&j_{n-4}^T\\-j_{n-4}&-j_{n-4}&-j_{n-4}&j_{n-4}&\left(\left(\lambda+1\right)I-J\right)_{n-4}\end{bmatrix}$.
\end{center}
Now, we apply finitely many elementary row and column operations on the matrix $\lambda I_n-A((K_n,H_1$ $(0,n-4)^-))$. First, subtracting the $5$-th row from all the lower rows and adding the $i$-th column to the $5$-th column, for $i=n,...,6$. This leads to the following matrix:
\begin{center}
$\lambda I_n-A((K_n,H_1(0,n-4)^-))=\begin{bmatrix}\lambda&1&1&1&4-n&*\\1&\lambda&1&-1&4-n&*\\1&1&\lambda&-1&4-n&*\\1&-1&-1&\lambda&n-4&*\\-1&-1&-1&1&\lambda-n+5&*\\\bf{0}&\bf{0}&\bf{0}&\bf{0}&\bf{0}&(\lambda+1)I_{n-5}\end{bmatrix}$.
\end{center}
Therefore, $\varphi((K_n,H_1(0,n-4)^-),\lambda)=(\lambda+1)^{n-4}(\lambda-1)(\lambda^3+(5-n)\lambda^2+(11-4n)\lambda+13n-57)$. Let $f_2(\lambda)=(\lambda+1)^2(\lambda-1)(\lambda^3+(5-n)\lambda^2+(11-4n)\lambda+13n-57)$. Set
\begin{center}
$h(\lambda)=P_{Q_3(1,n-5)}(\lambda)-f_2(\lambda)=8(\lambda+1)(\lambda-1)((n-7)\lambda+3n-17)$,
\end{center}
and the maximal solution of $h(\lambda)=0$ is $1$ for $n \geq 9$. And
\begin{center}
$P_{Q_3(1,n-5)}(\frac{n}{2}-1)=(\frac{n-4}{2})(\frac{n}{2})(-\frac{1}{16}n^4-\frac{1}{4}n^3+13n^2-44n-80)$.
\end{center}
Set $g(n)=-\frac{1}{16}n^4-\frac{1}{4}n^3+13n^2-44n-80$. It is easy to see that $g'(n)=-\frac14n^3-\frac34n^2+26n-44$, $g''(n)=-\frac{3}{4}n^2-\frac{3}{2}n+26$ and $g'''(n)=-\frac32n-\frac32$. Note that $g'''(n)<0$ for $n \geq 9$, then $g''(n)\le g''(9)=-\frac{193}{4}<0$, and thus $g'(n)\leq g'(9)=-53<0$. Hence, $g(n)$ is a monotone decreasing function for $n \geq 9$ and $g(n)\leq g(9)=-\frac{245}{16}<0$, and then $P_{Q_3(1,n-5)}(\frac{n}{2}-1)<0$. Hence, $\lambda_1(Q_3(1,n-5))>\frac n2-1>1$ for $n \geq 9$. Let $\lambda_1=\lambda_1(Q_3(1,n-5))$, so
\begin{center}
$-f_2(\lambda_1)=P_{Q_3(1,n-5)}(\lambda_1)-f_2(\lambda_1)>0.$
\end{center}
This indicates that $\lambda_1(A((K_n,H_1(1,n-5)^-)))<\lambda_1(A((K_n,H_1(0,n-4)^-)))$.

Finally, we will assert that $\lambda_1(A((K_n,H_1(n-5,1)^-)))>\lambda_1(A((K_n,H_1(0,n-4)^-)))$. Let $f_4(\lambda)=(\lambda+1)^2(\lambda-1)(\lambda^3+(5-n)\lambda^2+(11-4n)\lambda+13n-57)$. Note that
\begin{center}
	$h(\lambda)=f_4(\lambda)-P_{Q_3(n-5,1)}(\lambda)=8(n-5)(\lambda-1)^2(\lambda+1)$.
\end{center}
The maximal solution of $h(\lambda)=0$ is $1$ for $n \geq 9$, and
$\varphi((K_n,H_1(0,n-4)^-),\frac n2-1)=(\frac n2)^{n-4}(\frac{n-4}2)$ $(-\frac18n^3-\frac12n^2+18n-64)$. Similarly, we obtain that $f(n)=-\frac{1}{8}n^3-\frac{1}{2}n^2+18n-64<0$ for $n \geq 9$. Then $\varphi((K_n,H_1(0,n-4)^-),\frac n2-1)<0$. Hence, $\lambda_1(A((K_n,H_1(0,n-4)^-)))>\frac n2-1>1$ for $n \geq 9$. let $\lambda_1=\lambda_1(A((K_n,H_1(0,n-4)^-)))$, then
\begin{center}
	$-P_{Q_3(n-5,1)}(\lambda_1)=f_4(\lambda_1)-P_{Q_3(n-5,1)}(\lambda_1)>0$.
\end{center}
This indicates that $\lambda_1(A((K_n,H_1(n-5,1)^-)))>\lambda_1(A((K_n,H_1(0,n-4)^-)))$. Thus, $\lambda_1(A((K_n,$ $D_{1,n-3}^-)))>\lambda_1(A((K_n,H_1(s,t)^-)))$ by Case \ref{case1}.

Hence, $\lambda_1(A((K_n,$ $D_{1,n-3}^-)))>\lambda_1(A((K_n,H_1(s,t)^-)))$. The proof is completed.
\end{proof}
\begin{lemma}\label{y}
Let $s$ and $t$ be non-negative integers such that $s+t=n-5\geq 4$, and $H_2(s,t)$ be the graph depicted in $Fig. \ref{Fig2}$. Then 
	\begin{center}
		$\lambda_1(A((K_n,D_{1,n-3}^-)))>\lambda_1(A((K_n,H_2(s,t)^-)))$.
	\end{center}
\end{lemma}
\begin{proof}
Assume that $s, t \geq 1$. We give the $A((K_n,H_2(s,t)^-))$ and its corresponding quotient matrix $Q_4(s,t)$ by the vertex partition $V_1=\{v_1\}$, $V_2=\{v_2\}$, $V_3=\{v_3\}$, $V_4=\{v_4\}$, $V_5=\{v_5\}$, $V_6=N_{H_{2(s,t)}}(v_4)\backslash\{v_1,v_5\}$ and $V_7=N_{H_2(s,t)}(v_5)\backslash\{v_4\}$ as follows
\begin{center}
$\begin{gathered}A((K_n,H_2(s,t)^-))=\begin{bmatrix}0&-1&-1&-1&1&j_s^T&j_t^T\\-1&0&-1&1&1&j_s^T&j_t^T\\-1&-1&0&1&1&j_s^T&j_t^T\\-1&1&1&0&-1&-j_s^T&j_t^T\\1&1&1&-1&0&j_s^T&-j_t^T\\j_s&j_s&j_s&-j_s&j_s&\left(J-I\right)_s&J_{s\times t}\\j_t&j_t&j_t&j_t&-j_t&J_{t\times s}&\left(J-I\right)_t\end{bmatrix}\end{gathered}$,
\end{center}
and
\begin{center}
 $Q_4(s,t)=\begin{bmatrix}0&-1&-1&-1&1&s&t\\-1&0&-1&1&1&s&t\\-1&-1&0&1&1&s&t\\-1&1&1&0&-1&-s&t\\1&1&1&-1&0&s&-t\\1&1&1&-1&1&s-1&t\\1&1&1&1&-1&s&t-1\end{bmatrix}$.
\end{center}
Notice that the characteristic polynomial of $Q_4(s,t)$ is
\begin{align*}
P_{Q_{4(s,t)}}(\lambda)&=(\lambda-1)(\lambda+1)(\lambda^5+(7-n)\lambda^4-(6n-22)\lambda^3+(8st+4n+8t-30)\lambda^2\\&+(32st+22n-103)\lambda+13n-72st-40t-57.
\end{align*}
Adding $\alpha $J to the blocks of $A((K_n,H_2(s,t)^-))$, where $\alpha $ is constant. Then $A((K_n,H_2(s,t)^-))$ will be 
\begin{center}
$\begin{gathered}A_4=\begin{bmatrix}0&0&0&0&0&\bf{0}^T&\bf{0}^T\\0&0&0&0&0&\bf{0}^T&\bf{0}^T\\0&0&0&0&0&\bf{0}^T&\bf{0}^T\\0&0&0&0&0&\bf{0}^T&\bf{0}^T\\0&0&0&0&0&\bf{0}^T&\bf{0}^T\\\bf{0}&\bf{0}&\bf{0}&\bf{0}&\bf{0}&-I_s&\bf{0}\\\bf{0}&\bf{0}&\bf{0}&\bf{0}&\bf{0}&\bf{0}&-I_t\end{bmatrix}\end{gathered}$.
\end{center}
Since $\lambda_1(Q_4(s,t))>0$ and Spec$(A_4)$$=\begin{Bmatrix}-1^{[n-5]},0^{[5]}\end{Bmatrix}$, $\lambda_1(A((K_n,H_2(s,t)^-)))=\lambda_1(Q_4(s,t))$.\\Note that 
\begin{center}
$P_{Q_{4(s,t)}}(\lambda)-P_{Q_{4(s+1,t-1)}}(\lambda)=8((s-t+2)\lambda^2+(4s-4t+4)\lambda+9t-9s-14)$.
\end{center}
Next, we will discuss in two cases.

\begin{casess}\label{c3}
$s \geq t$.
\end{casess} 

Let $f(\lambda)=8((s-t+2)\lambda^2+(4s-4t+4)\lambda+9t-9s-14)$. Then $f'(\lambda)=16(s-t+2)\lambda+32(s-t+1)>0$ for $\lambda>0$, and $f(\lambda)$ is a monotone increasing function for $\lambda>0$.  We can obtain that $P_{Q_{4(s,t)}}(\frac{n}{2}-1)-P_{Q_{4(s+1,t-1)}}(\frac{n}{2}-1)>0$. And
\begin{align*}
P_{Q_{4(s,t)}}(\frac n2-1)=(\frac n2)(\frac{n-4}2)(-\frac1{32}n^5-\frac18n^4+\frac92n^3+(2st+2t-16)n^2+(8st-8t)n-96st-32t).
\end{align*}
Let $h(n)=(-\frac1{32}n^5-\frac18n^4+\frac92n^3+(2st+2t-16)n^2+(8st-8t)n-96st-32t)$.
Next, we will show $P_{Q_{4(s,t)}}(\frac n2-1)<0$. Thus, we just need to prove $h(n)<0$. Note that
\begin{align*}
&h'(n)=-\frac5{32}n^4-\frac12n^3+\frac{27}2n^2+2(2st+2t-16)n+8st-8t,\\
&h^{\prime\prime}(n)=-\frac58n^3-\frac32n^2+27n+2(2st+2t-16),\\
&h'''(n)=-\frac{15}{8}n^2-3n+27,\\
&h^4\big(n\big)=-\frac{15}{4}n-3<0.
\end{align*}
Recall that $s,t \geq 1$ such that $s+t=n-5$, then 
$$n-6\leq st\leq\begin{cases}(\frac{n-5}{2})^2&\text{if}~ n~is~odd,\\(\frac{n-4}{2})(\frac{n-6}{2})&\text{if}~ n~is~even.\end{cases}$$
Then $h'''(n)$ is a monotone decreasing for $n \geq 9$ and $h'''(n)\le h'''(9)=-\frac{1215}{8}<0$, and then $h''(n)$ is a monotone decreasing function for $n \geq 9$. One can have

$$h''(n)\leq\begin{cases}-\frac58n^3-\frac12n^2+19n-17&\text{if}~ n~is~odd,\\-\frac58n^3-\frac12n^2+19n-20&\text{if}~ n~is~even.\end{cases}$$
Set $h_1(n)=-\frac58n^3-\frac12n^2+19n-17$. Its derivative is $h_1^{\prime}(n)=-\frac{15}{8}n^2-n+19$ and $h_1^{\prime\prime}(n)=-\frac{15}4n-1$. Obviously, $h_1^{\prime\prime}(n)<0$ for $n\geq9$ and  $h_1^{\prime}(n)\leq h_1^{\prime}(9)=-\frac{1135}8<0$. Thus, $h_1(n)$ is a monotone decreasing for $n \geq 9$ and $h_1(n)\leq h_1(9)=-\frac{2737}8<0$. Note that $-\frac58n^3-\frac12n^2+19n-20<h_1(n)<0$. So, $h'(n)$ is monotone decreasing for $n \geq 9$ and 

$$h'(n)\leq\begin{cases}-\frac5{32}n^4+\frac12n^3+\frac{15}2n^2-41n+70&\text{if}~ n~is~odd,\\-\frac5{32}n^4+\frac12n^3+\frac{15}2n^2-44n+72&\text{if}~ n~is~even.\end{cases}$$
Set $h_2(n)=-\frac5{32}n^4+\frac12n^3+\frac{15}2n^2-41n+70$. Clearly, $h_2'(n)=-\frac{5}{8}n^3+\frac{3}{2}n^2+15n-41$, $h_2''(n)=-\frac{15}{8}n^2+3n+15$ and $h_2^{\prime\prime\prime}(n)=-\frac{15}4n+3$. Obviously, $h_2^{\prime\prime\prime}(n)<0$ for $n\geq9$. Thus, $h_2^{\prime\prime}(n)\leq h_2^{\prime\prime}(9)=-\frac{879}8<0$ and  $h_2^{\prime}(n)\leq h_2^{\prime}(9)=-\frac{1921}{8}<0$. Hence, $h_2(n)$ is a monotone decreasing for $n \geq 9$ and $h_2(n)\leq h_2(9)=-\frac{11269}{32}<0$. Note that $-\frac5{32}n^4+\frac12n^3+\frac{15}2n^2-44n+7<h_2(n)<0$. This implies that $h(n)$ is monotone decreasing for $n \geq 9$ and

$$h(n)\leq\begin{cases}-\frac1{32}n^5+\frac38n^4+\frac52n^3-\frac{113}2n^2+294n-520&\text{if}~ n~is~odd,\\-\frac1{32}n^5+\frac38n^4+\frac52n^3-58n^2+296n-480&\text{if}~ n~is~even.\end{cases}$$
Set $h_3(n)=-\frac1{32}n^5+\frac38n^4+\frac52n^3-\frac{113}2n^2+294n-520$. Similarly, we obtain that $h_3(n)<0$. Note that $-\frac1{32}n^5+\frac38n^4+\frac52n^3-58n^2+296n-480<h_3(n)<0$. This means that $\lambda_1(A((K_n,H_2(s,t)^-)))>\frac n2-1$. Let $\lambda_1=\lambda_1(A((K_n,H_2(s,t)^-)))$, thus
\begin{center}
$-P_{Q_{4(s+1,t-1)}}(\lambda_1)=P_{Q_{4(s,t)}}(\lambda_1)-P_{Q_{4(s+1,t-1)}}(\lambda_1)>0$.
\end{center}
Then $\lambda_1(A((K_n,H_2(s,t)^-)))<\lambda_1(A((K_n,H_2(s+1,t-1)^-)))$. By repeatedly using this operation, we can get that
\begin{center}
 $\lambda_1(A((K_n,H_2(s,t)^-))) \leq \lambda_1(A((K_n,H_2(n-6,1)^-)))$, 
\end{center}
the equality holds if and only if $s=n-6$, $t=1$.
Note that
\begin{align*}	
P_{Q_4(n-6,1)}(\lambda)&=(\lambda-1)(\lambda+1)(\lambda^5+(7-n)\lambda^4-(6n-22)\lambda^3+(12n-70)\lambda^2+(54n-295)\lambda\\&+335-59n).
\end{align*}
 Clearly, $H_2(n-5,0)\cong H_1(0,n-4)$. Then $\varphi((K_n,H_2(n-5,0)^-),\lambda)=(\lambda+1)^{n-4}(\lambda-1)(\lambda^3+(5-n)\lambda^2+(11-4n)\lambda+13n-57)$, and set $f_1(\lambda)=(\lambda+1)^{3}(\lambda-1)(\lambda^3+(5-n)\lambda^2+(11-4n)\lambda+13n-57)$ by Lemma \ref{a}. Hence
\begin{center}
$P_{Q_4(n-6,1)}(\lambda)-f_1(\lambda)=8(\lambda+1)(\lambda-1)((n-5)\lambda^2+4(n-6)\lambda+49-9n)$.
\end{center}
Let $f(\lambda)=(n-5)\lambda^2+4(n-6)\lambda+49-9n$. It is evidently that $f'(\lambda)=2(n-5)\lambda+4(n-6)>0$ for $n \geq 9$ and $\lambda>0$. Thus, $f(\lambda)$ is a monotone increasing function for $n \geq 9$ and $\lambda>0$, and $f(\frac n2-1)>0$ for $n \geq 9$. 
Notice that
\begin{align*}	
 P_{Q_4(n-6,1)}(\frac{n}{2}-1)=(\frac n2)(\frac{n-4}2)(-\frac{1}{32}n^5-\frac{1}{8}n^4+\frac{13}{2}n^3-18n^2-152n+544).
\end{align*}
Let $w(n)=-\frac{1}{32}n^5-\frac{1}{8}n^4+\frac{13}{2}n^3-18n^2-152n+544$. Similarly, we obtain that $w(n)<0$ for $n \geq 9$. Then $P_{Q_4(n-6,1)}(\frac{n}{2}-1)<0$ and  $\lambda_1(A((K_n,H_2(n-6,1)^-)))>\frac n2-1$. Let $\lambda_1=\lambda_1(A((K_n,H_2(n-6,1)^-)))$, then $f_1(\lambda_1)<0$. thus
\begin{center}
 $\varphi((K_n,H_2(n-5,0)^-),\lambda_1)<0$. 
\end{center}
This indicates that $\lambda_1(A((K_n,H_2(n-6,1)^-)))<\lambda_1(A((K_n,H_2(n-5,0)^-)))$. Since $H_2(n-5,0)\cong H_1(0,n-4)$,  $\lambda_1(A((K_n,H_2(s,t)^-)))<\lambda_1(A((K_n,D_{1,n-3}^-)))$ by Lemma \ref{a}.

\begin{casess}
 $s<t$. 
\end{casess}

\noindent{\bf{${\mbox{subcase 2.1. }}$}}
  $t-s=1$, i.e., $H_2(s,t)\cong H_2(\frac{n-5}{2},\frac{n-3}{2})$ and $P_{Q_{4(s,t)}}(\lambda)-P_{Q_{4(s+1,t-1)}}(\lambda)=8(\lambda^2-5)>0$ for $\lambda>\sqrt{5}$. Let $\lambda_1=\lambda_1(A((K_n,H_2(s,t)^-)))$. Note that $P_{Q_4(s,t)}(\frac{n}{2}-1)<0$ by Case \ref{c3}, then $\lambda_1>\frac{n}{2}-1>\sqrt{5}$ for $n \geq 9$. Hence, $P_{Q_{4(s+1,t-1)}}(\lambda_1)<0$ and $\lambda_1(A((K_n,H_2(\frac{n-5}{2},\frac{n-3}{2})^-)))<\lambda_1(A((K_n,H_2(\frac{n-3}{2},$ $\frac{n-5}{2})^-)))$. Since $\lambda_1(A((K_n,H_2(\frac{n-3}{2},\frac{n-5}{2})^-)))<\lambda_1(A((K_n,D_{1,n-3}^-)))$ by Case \ref{c3}, $\lambda_1(A((K_n,H_2(\frac{n-5}{2},$ $ \frac{n-3}{2})^-)))$ $<\lambda_1(A((K_n,D_{1,n-3}^-)))$.

\noindent{\bf{${\mbox{subcase 2.2. }}$}}
$t-s=2$, i.e., $H_2(s,t)\cong H_2(\frac{n-7}{2},\frac{n-3}{2})$. Set $h(\lambda)=P_{Q_{4(s,t)}}(\lambda)-P_{Q_{4(s-1,t+1)}}(\lambda)=8(2\lambda^2+12\lambda-22)$, and the maximal solution of $h(\lambda)=0$ is $-3+2\sqrt{5}$. Let $\lambda_1=\lambda_1(A((K_n,H_2(s,t$ $)^-)))$. Note that $P_{Q_4(s,t)}(\frac{n}{2}-1)<0$ by Case \ref{c3}, then $\lambda_1>\frac{n}{2}-1>-3+2\sqrt{5}$ for $n \geq 9$. This means that $P_{Q_{4(s-1,t+1)}}(\lambda_1)$ $<0$ and $\lambda_1(A((K_n,H_2(s,t)^-)))<\lambda_1(A((K_n,H_2(s-1,t+1)^-)))$. By repeatedly using this operation, we can get that
\begin{center}
	$\lambda_1(A((K_n,H_2(s,t)^-))) \leq \lambda_1(A((K_n,H_2(1, n-6)^-)))$,
\end{center}
the equality holds if and only if $s=1$, $t=n-6$.

\noindent{\bf{${\mbox{subcase 2.3. }}$}}
$t-s \geq 3$. Note that $P_{Q_{4(s,t)}}(\lambda)-P_{Q_{4(s-1,t+1)}}(\lambda)=8((t-s)\lambda^2+4(t-s+1)\lambda+9s-9t-4)$. Set $f(\lambda)=8((t-s)\lambda^2+4(t-s+1)\lambda+9s-9t-4)$, and then $f'(\lambda)=16(t-s)\lambda+32(t-s+1)>0$ for $\lambda>0$. Hence, $f(\lambda)$ is a monotone increasing function for $\lambda>0$ and $f(2)=8(3t-3s+4)>0$. Let $\lambda_1=\lambda_1(A((K_n,H_2(s,t)^-)))$. Note that  $P_{Q_4(s,t)}(\frac{n}{2}-1)<0$ by Case \ref{c3}, then $\lambda_1>\frac{n}{2}-1>2$ for $n \geq 9$. This indicates that $P_{Q_{4(s-1,t+1)}}(\lambda_1)<0$ and $\lambda_1(A((K_n,H_2(s,t)^-)))<\lambda_1(A((K_n,H_2(s-1,t+1)^-)))$. By repeatedly using this operation, one can have
\begin{center}
	$\lambda_1(A((K_n,H_2(s,t)^-))) \leq \lambda_1(A((K_n,H_2(1, n-6)^-)))$,
\end{center}
the equality holds if and only if $s=1$, $t=n-6$.

Next, we will give that $\lambda_1(A((K_n,H_2(1,n-6)^-)))<\lambda_1(A((K_n,H_2(0,n-5)^-)))$. Note that
\begin{align*}
P_{Q_4(1,n-6)}(\lambda)&=(\lambda-1)(\lambda+1)(\lambda^5+(7-n)\lambda^4-(6n-22)\lambda^3+(20n-126)\lambda^2\\&+(54n-295)\lambda+615-99n).
\end{align*}
By appropriatly marking the vertices of $(K_n,H_2(0,n-5)^-)$, we can get that
\begin{center}
$\begin{gathered}A((K_n,H_2(0,n-5)^-))=\begin{bmatrix}0&-1&-1&-1&1&j_{n-5}^T\\-1&0&-1&1&1&j_{n-5}^T\\-1&-1&0&1&1&j_{n-5}^T\\-1&1&1&0&-1&j_{n-5}^T\\1&1&1&-1&0&-j_{n-5}^T\\j_{n-5}&j_{n-5}&j_{n-5}&j_{n-5}&-j_{n-5}&\left(J-I\right)_{n-5}\end{bmatrix}\end{gathered}$.
\end{center}
Then 
\begin{center}
$\begin{gathered}\lambda I_n-A((K_n,H_2(0,n-5)^-))=\begin{bmatrix}\lambda&1&1&1&-1&-j_{n-5}^T\\1&\lambda&1&-1&-1&-j_{n-5}^T\\1&1&\lambda&-1&-1&-j_{n-5}^T\\1&-1&-1&\lambda&1&-j_{n-5}^T\\-1&-1&-1&1&\lambda&j_{n-5}^T\\-j_{n-5}&-j_{n-5}&-j_{n-5}&-j_{n-5}&j_{n-5}&\left(\left(\lambda+1\right)I-J\right)_{n-5}\end{bmatrix}\end{gathered}$.
\end{center}
Now, we apply finitely many elementary row and column operations on the matrix $\lambda I_n-A((K_n,H_2$ $(0,n-5)^-))$. First, subtracting the $6$-th row from all the lower rows and adding the $i$-th column to the $6$-th column, for $i=n,...,7$. This leads to the following matrix:
\begin{center}
$\lambda I_n-A((K_n,H_2(0,n-5)^-))=\begin{bmatrix}\lambda&1&1&1&-1&5-n&*\\1&\lambda&1&-1&-1&5-n&*\\1&1&\lambda&-1&-1&5-n&*\\1&-1&-1&\lambda&1&5-n&*\\-1&-1&-1&1&\lambda&n-5&*\\-1&-1&-1&-1&1&\lambda+6-n&*\\\bf{0}&\bf{0}&\bf{0}&\bf{0}&\bf{0}&\bf{0}&(\lambda+1)I_{n-6}\end{bmatrix}$.
\end{center}
Thus, the characteristic polynomial of $A((K_n,H_2(0,n-5)^-))$ is
\begin{align*}
\varphi((K_n,H_2(0,n-5)^-),\lambda)&=(\lambda+1)^{n-6}(\lambda-1)^2(\lambda^4+(8-n)\lambda^3+(30-7n)\lambda^2+(5n-40)\lambda\\&+27n-143).
\end{align*}
Let $f_2(\lambda)=(\lambda+1)(\lambda-1)^2(\lambda^4+(8-n)\lambda^3+(30-7n)\lambda^2+(5n-40)\lambda+27n-143)$. Notice that
\begin{center}
$P_{Q_4(1,n-6)}(\lambda)-f_2(\lambda)=8(\lambda-1)(\lambda+1)((n-7)\lambda^2+4(n-6)\lambda+59-9n)$.
\end{center}
Let $g(\lambda)=(n-7)\lambda^2+4(n-6)\lambda+59-9n$. Then $g'(\lambda)=2(n-7)\lambda+4(n-6)>0$ for $n \geq 9$, $\lambda>0$ and $g(\lambda)$ is a monotone increasing function for  $n \geq 9$ and $\lambda>0$. Note that $g(\frac{n}{2}-1)=\frac{1}{4}n^3-\frac{3}{4}n^2-17n+76$, thus $g'(\frac{n}{2}-1)=\frac{3}{4}n^2-\frac{3}{2}n-17$ and $g^{\prime\prime}(\frac n2-1)=\frac32n-\frac32>0$ for $n \geq 9$. Hence, $g'(\frac{n}{2}-1)\geq\frac{121}{4}>0$, and $g(\frac n2-1)\geq\frac{89}2>0$. Notice that
\begin{center}
$P_{Q_4(1,n-6)}(\frac{n}{2}-1)=(\frac{n}{2})(\frac{n-4}{2})(-\frac{1}{32}n^5-\frac{1}{8}n^4+\frac{17}{2}n^3-40n^2-128n+768)$.
\end{center}
Let $z(n)=-\frac{1}{32}n^5-\frac{1}{8}n^4+\frac{17}{2}n^3-40n^2-128n+768$. It is easy to see that $z'(n)=-\frac5{32}n^4-\frac12n^3+\frac{51}2n^2-80n-128$, 
$z^{\prime\prime}(n)=-\frac58n^3-\frac32n^2+51n-80$, $z^{\prime\prime\prime}(n)=-\frac{15}8n^2-3n+51$ and $z^{(4)}(n)=-\frac{15}{4}n-3<0$ for $n \geq 9$. Hence, $z^{\prime\prime\prime}(n)\leq z^{\prime\prime\prime}(9)=-\frac{1023}8<0$ and $z''(n)\leq z''(9)=-\frac{1585}{8}<0$. So, $z'(n)\leq z'(9)=-\frac{5509}{32}<0$ and $z(n)$ is a monotone decreasing for $n \geq 9$, and then $z(n)\leq z(9)=-\frac{2973}{32}<0$. Since $P_{Q_4(1,n-6)}(\frac{n}{2}-1)<0$, $\lambda_1(A((K_n,H_2(1,n-6)^-)))>\frac n2-1$. Let $\lambda_1=\lambda_1(A((K_n,H_2(1,n-6)^-)))$. Hence, $f_2(\lambda_1)<0$ and $\lambda_1(A((K_n,H_2(1,n-6)^-)))<\lambda_1(A((K_n,H_2(0,n-5)^-)))$.
 
Finally, we assert that $\lambda_1(A((K_n,H_2(0,n-5)^-)))<\lambda_1(A((K_n,H_2(n-5,0)^-)))$. let $f_3(\lambda)=(\lambda+1)^2(\lambda-1)(\lambda^3+(5-n)\lambda^2+(11-4n)\lambda+13n-57)$ and $f_4(\lambda)=(\lambda-1)^2(\lambda^4+(8-n)\lambda^3+(30-7n)\lambda^2+(5n-40)\lambda+27n-143)$. Set
 \begin{center}
 $h(\lambda)=f_4(\lambda)-f_3(\lambda)=8(n-5)(\lambda-1)(\lambda^2-5)$,
 \end{center}
and the maximal solution of $h(\lambda)=0$ is $\sqrt{5}$. And
\begin{center}
$f_4(\frac{n}{2}-1)=(\frac{n-4}{2})^{2}(-\frac{1}{16}n^{4}-\frac{1}{2}n^{3}+11n^{2}-24n-80)$.
\end{center}
Let $y(n)=-\frac{1}{16}n^4-\frac{1}{2}n^3+11n^2-24n-80$. Similarly, we obtain that  $y(n)<0$ for $n \geq 9$. Then $f_4(\frac{n}{2}-1)<0$ and $\lambda_1(A((K_n,H_2(0,n-5)^-)))>\frac n2-1>\sqrt{5}$ for $n \geq 9$. Let  $\lambda_1=\lambda_1(A((K_n,H_2(0,n-5)^-)))$, then $f_3(\lambda_1)<0$. This induces that  $\lambda_1(A((K_n,H_2(0,n-5)^-)))<\lambda_1(A((K_n,H_2(n-5,0)^-)))$. 

Obviously, $\lambda_1(A((K_n,H_2(n-5,0)^-)))=\lambda_1(A((K_n,H_1(0,n-4)^-)))<\lambda_1(A((K_n,D_{1,n-3}^-)))$ by Lemma \ref{a}. Hence, the proof is completed.
\end{proof}
\section{Proof of Theorem \ref{thm1}}\label{se3}
\begin{lemma}\label{t}\cite{12}.
	Let $T$ be a spanning tree of $K_n$ and $n \geq 6$. If  $\Gamma=\begin{pmatrix}K_n,T^-\end{pmatrix}$ is an unbalanced signed complete graph with maximum index, then $T \cong D_{1,n-3}$.
\end{lemma}
\begin{lemma}\label{o}\cite{11}.
	Let $r,s,t$ and $u$ be distinct vertices of a signed graph $\Gamma$ and  $X=\begin{pmatrix}x_1,x_2,...,x_n\end{pmatrix}^T$ be an eigenvector corresponding to  $\lambda_1(A(\Gamma))$. Then

	(i) Let $\Gamma^{\prime}$ be obtained from $\Gamma$ by reversing the sign of the positive edge $rs$ and the negative edge $rt$. If
	$$\begin{cases}x_r\geq0,x_s\leq x_t&or\\x_r\leq0,x_s\geq x_t,&\end{cases}$$
	then $\lambda_1(A(\Gamma))\leq\lambda_1(A(\Gamma^\prime))$. If at least one inequality for the entries of $X$ is strict, then $\lambda_1(A(\Gamma))<\lambda_1(A(\Gamma^\prime))$.
	
	(ii) Let $\Gamma^{\prime}$ be obtained from $\Gamma$ by reversing the sign of the positive edge $rs$ and the nagative edge $tu$. If $x_rx_s\leq x_tx_u$, then  $\lambda_1(A(\Gamma))\leq\lambda_1(A(\Gamma^\prime))$. If at least one of the entries $x_r,x_s,x_t,x_u$ is distinct from zero, then $\lambda_1(A(\Gamma))<\lambda_1(A(\Gamma^\prime))$.
\end{lemma}
Let $R(r,s,t)$ stand for relocation $(i)$ in the Lemma \ref{o}. Kafai, Heydari, Rad and Maghasedi \cite{999}  showed that among all signed complete graphs
of order $n>5$ whose negative edges induce a unicyclic graph of order
$k$ and maximizes the index, the negative edges induce a triangle with
all remaining vertices being pendant at the same vertex of the triangle.
\begin{lemma}\label{77}
Let $H$ be a connected $K_{2,2}$-minor free spanning subgraph of $K_n$ for $n\geq 9$. If $\Gamma=(K_n,H^-)$ has the maximum index, then $H \cong U_1$.
\end{lemma}
\begin{proof}
 It is clear that $H$ is either a tree or a connected graph that only contains $3$-length cycles. According to Lemmas \ref{z} and \ref{t}, $H$ is a connected graph which contains $3$-length cycles. Without loss of generality, suppose that $C_3=v_1v_2v_3v_1$ is a $3$-length cycle of $H$. Let $X=(x_1,x_2,...,x_n)^T$ be a unit eigenvector associated with  $\lambda_1(A(\Gamma))$. Note that $-X$ must be a unit eigenvector of $\Gamma $ if $X$ is a unit eigenvector.

\begin{claim}\label{cc1}
There exists an integer $i$ such that $x_i\neq 0$ for $1 \leq i \leq3$.\label{11}
\end{claim}
Otherwise, $x_1=x_2=x_3=0$. Due to $H$ is a connected graph for $n \geq 9$, there is a vertex $v_4\in V(H)\backslash\ V(C_3)$. Without loss of generality, assume that $v_3v_4\in E(H)$. Firstly, let $x_{4}\neq0$. If $v_4$ is not contained in any $3$-length cycle, then the relocation $R(v_4,v_{3},v_1)$ contradicts with the maximality of $\lambda_1(A(\Gamma))$ by Lemma \ref{o}. If there is another $3$-length cycle $C_3^{\prime}$ such that $v_4\in V(C_3^{\prime})$, then we will consider three cases. If $V(C_3)\cap V(C_3^{\prime})=\phi $, then the relocation $R(v_4,v_{3},v_1)$ contradicts with the maximality of $\lambda_1(A(\Gamma))$ by Lemma \ref{o}. If $V(C_3)\cap V(C_3^{\prime})=\{v_3\}$, then we can construct a new unbalanced signed complete graph $\Gamma^{\prime}$ whose negative edge-induced connected subgraph is still $K_{2,2}$-minor free by reversing the sign of the negative edge $v_3v_4$ such that $\lambda_1(\Gamma^{\prime})>\lambda_1(\Gamma)$, a contradiction. If $|V(C_3)\cap V(C_3^{\prime})|=2$, then $H$ contains $K_{2,2}$-minor, a contradiction. Hence,  $x_{4}=0$. By repeatedly conducting similar discussion on the vertices in $V(H)\backslash\{v_1,v_2,v_3,v_4\}$, we have $X=\bf{0}$, a contradiction.

\begin{claim}\label{claim1.1}
$H$ contains only one $3$-length cycle.
\end{claim}
Otherwise, let $V(C^{\prime}_3)=v_4v_5v_6v_4$ be another $3$-length cycle of $H$. Then we divide the Claim \ref{claim1.1} into the following three cases.

\begin{casesss}\label{s1}
There exists two vertices $v_i,v_j \in  V(C_3^{\prime})$ such that $x_ix_j>0$.
\end{casesss}

Without loss of generality, assume that $x_{4}x_{5}>0$. Then we can construct a new unbalanced signed complete graph $\Gamma^{\prime}$ whose negative edge-induced connected subgraph is still $K_{2,2}$-minor free by reversing the sign of the negative edge $v_{4}v_{5}$ such that
\begin{center}
$\lambda_1(A(\Gamma^{\prime}))-\lambda_1(A(\Gamma))\geq X^T(A(\Gamma^{\prime})-A(\Gamma))X=4x_4x_5>0$,
\end{center}
that is, $\lambda_1(A(\Gamma^{\prime}))>\lambda_1(A(\Gamma))$, a contradiction.

\begin{casesss}\label{s2}
There exists only one vertex $v_i \in V(C_3^{\prime})$ such that $x_i=0$.
\end{casesss}

Without loss of generality, assume that $x_4=0$. By Case \ref{s1}, $x_5x_6<0$. Then we can construct a new unbalanced signed complete graph  $\Gamma^{\prime}$  whose negative edge-induced connected subgraph is still $K_{2,2}$-minor free by reversing the sign of the negative edge $v_{4}v_{5}$ such that
\begin{center}
$\lambda_1(A(\Gamma^{\prime}))-\lambda_1(A(\Gamma))\geq X^T(A(\Gamma^{\prime})-A(\Gamma))X=4x_4x_5=0.$
\end{center}
If $\lambda_1(A(\Gamma^{\prime}))=\lambda_1(A(\Gamma))$, then $X$ is also a unit eigenvector  of $A(\Gamma^{\prime})$ corresponding to $\lambda_1(A(\Gamma^{\prime}))$. However,
\begin{center}
	$\lambda_1(A(\Gamma^{\prime}))x_4-\lambda_1(A(\Gamma))x_4=2x_5\neq0$.
\end{center}
a contradiction. 

\begin{casesss}\label{s3}
There exists two vertices $v_i, v_j \in  V(C_3^{\prime})$ such that $x_i=x_j=0$.
\end{casesss}

Without loss of generality, assume that $x_4=x_5=0$. By Claim \ref{cc1}, $x_6 \neq 0$. Then we can construct a new unbalanced signed complete graph $\Gamma^{\prime}$ whose negative edge-induced connected subgraph is still $K_{2,2}$-minor free by reversing the sign of the negative edge $v_{4}v_{6}$ such that
\begin{center}
$\lambda_1(A(\Gamma^{\prime}))-\lambda_1(A(\Gamma))\geq X^T(A(\Gamma^{\prime})-A(\Gamma))X=4x_4x_6=0.$
\end{center}
If $\lambda_1(A(\Gamma^{\prime}))=\lambda_1(A(\Gamma))$, then $X$ is also a unit eigenvector  of $A(\Gamma^{\prime})$ corresponding to $\lambda_1(A(\Gamma^{\prime}))$. However,
\begin{center}
	$\lambda_1(A(\Gamma^{\prime}))x_4-\lambda_1(A(\Gamma))x_4=2x_6\neq0$.
\end{center}
a contradiction. 

Claim \ref{claim1.1} means that $H$ contains only one $3$-length cycle $C_3=v_1v_2v_3v_1$. Then $H \cong U_1$ by \cite{999}. So, the proof is completed.
\end{proof}
\begin{lemma}\label{l9}
Let $H$ be a connected $K_{2,2}$-minor free spanning subgraph of $K_n$ for $n\geq 9$. If $\Gamma=(K_n,H^-)$ is not switching isomorphic to $(K_n,U_1^-)$ and has the maximum index, then $H\cong D_{1,n-3}$.
\end{lemma}
\begin{proof}
Recall that $H$ is either a tree or a connected  graph that only contains $3$-length cycles. Next, we will assert that $H$ must be a tree. Otherwise, Let $C_3=v_1v_2v_3v_1$ be a $3$-length cycle of $H$ and $X=(x_1,x_2,...,x_n)^T$ be a unit eigenvector associated with  $\lambda_1(A(\Gamma))$. Note that $-X$ must be a unit eigenvector of $\Gamma $ if $X$ is a unit eigenvector. Then we divide the proof into the following two cases. Firstly, we consider that $H$ contains only one $3$-length cycle $C_3=v_1v_2v_3v_1$.
\begin{claimm}\label{claim1.2}
	$C_3$ has only one vertex $v_i$ such that $d_H(v_i) \geq 3$ for $1 \leq i \leq 3$.
\end{claimm}
Otherwise, without loss of generality, assume that $d_H(v_{1})\geq3$ and $d_H(v_{2})\geq3$. Let $v_p\in N_H(v_2)\backslash\{v_1,v_3\}$ and $v_{p^{\prime}}\in N_H(v_1)\backslash\{v_2,v_3\}$. Then we will divide into the following five cases.

\begin{casessss}
	$x_{1}=x_{2}=0$.
\end{casessss}

We first assert that $x_{p^{\prime}}=0$. Otherwise, the relocation $R(v_{p^{\prime}},v_2,v_1)$ contradicts with the maximality of $\lambda_1(A(\Gamma))$ by Lemma \ref{o}. Secondly, we assert that $x_{3}=0$. Otherwise, we can construct a new unbalanced signed graph $\Gamma^{\prime}$ whose negative edge-induced connected subgraph is still $K_{2,2}$-minor free by reversing the sign of the positive edge $v_{3}v_{p^{\prime}}$ and the negative edge $v_{1}v_{2}$ such that  $\lambda_1(A(\Gamma))<\lambda_1(A(\Gamma^\prime))$ by $(ii)$ of Lemma \ref{o}, a contradiction. However, similar to Claim \ref{11} of Lemma \ref{77}, this will lead to  $X=\bf{0}$, a contradiction.

\begin{casessss}
	$x_1=x_2\neq 0$.
\end{casessss}
Without loss of generality, assume that $x_1=x_2>0$. At first, we assert that $x_p=0$. Otherwise, the relocation $R(v_{p},v_1,v_2)$ contradicts with the maximality of $\lambda_1(A(\Gamma))$ by Lemma \ref{o}. Next, we claim that $x_{3}<0$. Otherwise, the relocation $R(v_3,v_p,v_1)$ contradicts with the maximality of $\lambda_1(A(\Gamma))$ by Lemma \ref{o}. Hence, we can construct a new unbalanced signed complete graph $\Gamma^{\prime}$ whose negative edge-induced connected subgraph is still $K_{2,2}$-minor free by reversing the sign of the positive edge $v_{p}v_{3}$ and the negative edge $v_{1}v_{2}$  such that  $\lambda_1(A(\Gamma))<\lambda_1(A(\Gamma^\prime))$ by $(ii)$ of Lemma \ref{o}, a contradiction.

\begin{casessss}
	$x_1\neq x_2$ and $x_1x_2<0$.
\end{casessss}

Without loss of generality, assume that $x_1>0$, $x_2<0$. Firstly, we assert that $x_p>0$. Otherwise, the relocation $R(v_{p},v_1,v_2)$ contradicts with the maximality of $\lambda_1(A(\Gamma))$ by Lemma \ref{o}. If $x_{p^{\prime}}\leq  x_{2}$, then the relocation $R(v_{p},v_{p^{\prime}},v_2)$ contradicts with the maximality of $\lambda_1(A(\Gamma))$ by Lemma \ref{o}. So,  $x_{p^{\prime}}>  x_{2}$. Secondly, we claim that $x_3>0$. Otherwise, the relocation $R(v_{3},v_{p^{\prime}},v_2)$ contradicts with the maximality of $\lambda_1(A(\Gamma))$ by Lemma \ref{o}. Next, we assert that $x_{p^{\prime}}<0$. Otherwise, we can construct a new unbalanced signed complete graph $\Gamma^{\prime}$ whose negative edge-induced connected subgraph is still $K_{2,2}$-minor free by reversing the sign of the positive edge $v_{2}v_{p^{\prime}}$ and the negative edge $v_{1}v_{3}$ such that $\lambda_1(A(\Gamma))<\lambda_1(A(\Gamma^\prime))$ by $(ii)$ of Lemma \ref{o}, a contradiction. Finally, we assert that $x_p<x_1$. Otherwise, the relocation $R(v_{p^{\prime}},v_p,v_1)$ contradicts with the maximality of $\lambda_1(A(\Gamma))$ by Lemma \ref{o}. However, the relocation $R(v_{3},v_{p},v_1)$ contradicts with the maximality of $\lambda_1(A(\Gamma))$ by Lemma \ref{o}.

\begin{casessss}
	$x_1\neq x_2$ and $x_1x_2=0$.
\end{casessss}
Without loss of generality, assume that $x_1>0$, $x_2=0$. At first, we assert that $x_p>0$. Otherwise, the relocation $R(v_{p},v_{1},v_2)$ contradicts with the maximality of $\lambda_1(A(\Gamma))$ by Lemma \ref{o}. Next, we claim that $x_3>0$. Otherwise, the relocation $R(v_{p},v_{3},v_2)$ contradicts with the maximality of $\lambda_1(A(\Gamma))$ by Lemma \ref{o}. However, we can construct a new unbalanced signed complete graph $\Gamma^{\prime}$ whose negative edge-induced connected subgraph is still $K_{2,2}$-minor free by reversing the sign of the positive edge $v_{2}v_{{p}^{\prime}}$ and the negative edge $v_{1}v_{3}$ such that $\lambda_1(A(\Gamma))<\lambda_1(A(\Gamma^\prime))$ by $(ii)$ of Lemma \ref{o},  a contradiction.

\begin{casessss}
	$x_1\neq x_2$ and $x_1x_2>0$.
\end{casessss}
Without loss of generality, assume that $x_1>x_2>0$. Firstly, we claim that $x_p>0$. Otherwise, the relocation $R(v_{p},v_{1},v_2)$ contradicts with the maximality of $\lambda_1(A(\Gamma))$ by Lemma \ref{o}. Secondly, we assert that $x_3> x_{2}$. Otherwise, the relocation $R(v_{p},v_{3},v_2)$ contradicts with the maximality of $\lambda_1(A(\Gamma))$ by Lemma \ref{o}. Finally,  we claim that $x_{p^{\prime}}> x_{2}$. Otherwise, the relocation $R(v_{3},v_{p^{\prime}},v_2)$ contradicts with the maximality of $\lambda_1(A(\Gamma))$ by Lemma \ref{o}. However, the relocation $R(v_{p^{\prime}},v_2,v_1)$ contradicts with the maximality of $\lambda_1(A(\Gamma))$ by Lemma \ref{o}.

By Claim \ref{claim1.2}, without loss of generality, it can be assume that $d_H(v_{1})\geq 3$, $d_H(v_{2})=d_H(v_{3})=2$. Since $H\ncong U_1$, there is a vertex $v_s\in N_H(v_1)\backslash\{v_2,v_3\}$ such that $d_H(v_s)\geq 2$. 
\begin{claimm}\label{claim1.3}
$H \cong D_{1,n-3}$.
\end{claimm}
Without loss of generality, assume that $v_t\in N_H(v_s)\backslash$ $\{v_1\}$. Then we divide the proof into the following two cases.
\begin{case1}
	$x_t=0$.
\end{case1}
Firstly, we assert that $x_1= x_s$. Otherwise, the relocation $R(v_{t},v_{1},v_s)$ contradicts with the maximality of $\lambda_1(A(\Gamma))$ by Lemma \ref{o}. Similarly, we have $x_2=x_3=x_s$. Next, we claim that $x_1=x_2=x_3=x_s=0$. Otherwise, we can construct a new unbalanced signed complete graph $\Gamma^{\prime}$ whose negative edge-induced connected subgraph is still $K_{2,2}$-minor free by reversing the sign of the positive edge $v_{t}v_{1}$ and the negative edge $v_{2}v_{3}$ such that $\lambda_1(A(\Gamma))<\lambda_1(A(\Gamma^\prime))$ by $(ii)$ of Lemma \ref{o}, a contradiction. However, similar to Claim \ref{11} of Lemma \ref{77}, this will lead to $X=\bf{0}$, a contradiction.
\begin{case1}
	$x_t\neq 0$.
\end{case1}
Without loss of generality, assume that $x_t>0$. At first, we assert that $x_1>x_s$. Otherwise, the relocation $R(v_{t},v_{1},v_s)$ contradicts with the maximality of $\lambda_1(A(\Gamma))$ by Lemma \ref{o}. Similarly, we have $x_3>x_s$ and $x_2>x_s$. If $x_3\geq0$, then the relocation $R(v_{3},v_{s},v_2)$ contradicts with the maximality of $\lambda_1(A(\Gamma))$ by Lemma \ref{o}. So, $x_3<0$. Similarly, $x_2<0$. We also claim that $x_1>0$. Otherwise, the relocation $R(v_{1},v_{t},v_2)$ contradicts with the maximality of $\lambda_1(A(\Gamma))$ by Lemma \ref{o}. Finally, we assert that $x_s<0$. Otherwise, the relocation $R(v_s,v_2,v_1)$ contradicts with the maximality of $\lambda_1(A(\Gamma))$ by Lemma \ref{o}. 

After the above preparations, we will further discuss in two subcases.

\begin{subcase1}
	$d_H(v_1)>3$.
\end{subcase1}

Without loss of generality, assume that $v_w\in N_H(v_1)\backslash\{v_2,v_3,v_s\}$. At first, we assert that $x_w<0$. Otherwise, the relocation $R(v_{w},v_{2},$ $v_1)$ contradicts with the maximality of $\lambda_1(A(\Gamma))$ by Lemma \ref{o}. If $x_w\leq x_s$, then the relocation $R(v_{t},v_{w},$ $v_s)$ contradicts with the maximality of $\lambda_1(A(\Gamma))$ by Lemma \ref{o}. So, $x_w> x_s$. We also claim that $x_t<x_1$. Otherwise, the relocation $R(v_{w},v_{t},$ $v_1)$ contradicts with the maximality of $\lambda_1(A(\Gamma))$ by Lemma \ref{o}. Next, we assert that $d_H(v_{t})=1$. Otherwise, Let $v_a\in N_{H}(v_t)\setminus\{v_s\}$. If $x_a\leq0$, then the relocation $R(v_{a},v_{1},$ $v_t)$ contradicts with the maximality of $\lambda_1(A(\Gamma))$ by Lemma \ref{o}.  If $x_a>0$, then the relocation $R(v_{a},v_{s},$ $v_t)$ contradicts with the maximality of $\lambda_1(A(\Gamma))$ by Lemma \ref{o}. Similarly, $d_H(u)=1$ for any $u \in N_{H}(v_s)\setminus\{v_1\}$. Finally, we claim that $d_H(v_{w})=1$. Otherwise, let $v_b\in N_{H}(v_w)\setminus\{v_1\}$. If $x_b\leq0$, then the relocation $R(v_{b},v_{1},$ $v_w)$ contradicts with the maximality of $\lambda_1(A(\Gamma))$ by Lemma \ref{o}. If $x_b>0$, then the relocation $R(v_{b},v_{s},$ $v_w)$ contradicts with the maximality of $\lambda_1(A(\Gamma))$ by Lemma \ref{o}. In fact, $d_H(u)=1$ for all $u\in N_{H}(v_1)\setminus\{v_2,v_3,v_s\}$. Thus, $(K_n,H^-)$ is switching isomorphic to $(K_n,H_1(s,t)^-)$ for $t\geq 1$. However,  $\lambda_1(A((K_n,H_1(s,t)^-)))<\lambda_1(A((K_n,D_{1,n-3}^-)))$ for $t \geq 1$ by the proof of Lemma \ref{a}. 

\begin{subcase1}
	$d_H(v_1)=3$.
\end{subcase1}

Firstly, we assert that the length of the longest path passing through $v_s$ and $v_t$ starting $v_1$ is at most $3$. Otherwise, Let  $v_{t_1}\in N_{H}(v_t)\setminus\{v_s\}$ and $v_{t_2}\in N_{H}(v_{t_1})\setminus\{v_t\}$. If $x_s\geq x_{t_1}$, then the relocation $R(v_{1},v_{t_1},$ $v_s)$ contradicts with the maximality of $\lambda_1(A(\Gamma))$ by Lemma \ref{o}. So, $x_s< x_{t_1}$.  If $x_{t_1}\geq0$, then the relocation $R(v_{t_1},v_{s},$ $v_t)$ contradicts with the maximality of $\lambda_1(A(\Gamma))$ by Lemma \ref{o}. So, $x_{t_1}<0$. If $x_{t_2} \leq 0$, then the relocation $R(v_{t_2},v_t,v_{t_1})$ contradicts with the maximality of $\lambda_1(A(\Gamma))$ by Lemma \ref{o}. Thus, $x_{t_2} >0$. However, the relocation $R(v_{t_2},v_s,v_{t_1})$ contradicts with the maximality of $\lambda_1(A(\Gamma))$ by Lemma \ref{o}, a contradiction. This implies that in the case where exists $u\in N_H(v_t)\setminus\{v_s\}$, we have $d_H({u})=1$. Similar to the above discussion, we will next consider $d_H({v_s})$.

At first, we assume that $d_H(v_s)\geq3$. If $d_H(u)=1$ for any $u\in N_{H}(v_s)\setminus\{v_1\}$, then $(K_n,H^-)$ is switching isomorphic to $(K_n,H_2(n-5,0)^-)$. However, $\lambda_1(A((K_n,H_2(s,t)^-)))<\lambda_1(A((K_n,D_{1,n-3}^-$ $)))$ by the proof of Lemma \ref{y}. If there exists at least two vertices in $N_{H}(v_s)\setminus\{v_1\}$ whose degree is greater than or equal to $2$. Without loss of generality, assume that $d_H(v_t)\geq 2$ and $d_H(v_b) \geq 2$ for $v_b\in N_{H}(v_s)\setminus\{v_1\}$. Then let $v_a \in N_{H}(v_t)\setminus\{v_s\}$ and $v_c \in N_{H}(v_b)\setminus\{v_s\}$. Firstly, we assert that $x_a<0$. Otherwise, the relocation $R(v_a,v_{s},v_t)$ contradicts with the maximality of $\lambda_1(A(\Gamma))$ by Lemma \ref{o}.
Secondly, we claim that $x_{b}>0$. Otherwise, the relocation $R(v_b,v_{t},v_s)$ contradicts with the maximality of $\lambda_1(A(\Gamma))$ by Lemma \ref{o}.
If $x_{c}\geq 0$, then the relocation  $R(v_c,v_{s},v_{b})$ contradicts with the maximality of $\lambda_1(A(\Gamma))$ by Lemma \ref{o}. So, $x_{c}< 0$. 
Finally, we also assert that $x_{t}<x_b$. Otherwise, the relocation $R(v_{c},v_t,v_{b})$ contradicts with the maximality of $\lambda_1(A(\Gamma))$ by Lemma \ref{o}. However, the relocation  $R(v_a,v_{b},v_{t})$ contradicts with the maximality of $\lambda_1(A(\Gamma))$ by Lemma \ref{o}. This implies that 
there is only one vertex in $N_{H}(v_s)\setminus\{v_1\}$ whose degree is greater than or equal to $2$. Without loss of generality, assume that  $d_H(v_t)\geq 2$. Then $d_H({u})=1$ for all $u\in N_{H}(v_t)\setminus\{v_s\}$ through the previous discussion. Thus, $(K_n,H^-)$ is switching isomorphic to $(K_n,H_2(s,t)^-)$. However, $\lambda_1(A((K_n,H_2(s,t)^-)))<\lambda_1(A((K_n,D_{1,n-3}^-)))$ by the proof of Lemma \ref{y}. Secondly, assume that $d_H(v_s)=2$.
Then $(K_n,H^-)$ is switching isomorphic to $(K_n,H_2(0,n-5)^-)$.
However, $\lambda_1(A((K_n,H_2(0,n-5)^-)))<\lambda_1(A((K_n,D_{1,n-3}^-)))$ by the proof of Lemma \ref{y}.  

Next, we consider that $H$ contains at least two $3$-length cycles. By similar arguments as in the proof of Claim \ref{cc1} of Lemma \ref{77}, there exists a vertex $u$ of $3$-length cycle such that $x_u\neq 0$. Now, we assert that $H$ contians at most two $3$-length cycles. Otherwise, we can construct a new unbalanced signed complete graph $\Gamma^{\prime}=(K_n,H_1^-)$ such that  $\lambda_1(A(\Gamma))<\lambda_1(A(\Gamma^\prime))$ by similar arguments as in the proof of Claim \ref{claim1.1} of Lemma \ref{77}, where $H_1$ is $K_{2,2}$-minor free graph and contains two $3$-length cycles. This also contradicts the maximality of $\lambda_1(A(\Gamma))$. Thus, $H$ contains exactly two $3$-length cycles. If $H \ncong  U_2$, then we can construct a new unbalanced signed complete graph $\Gamma^{\prime\prime}=(K_n,H_2^-)$ such that $\lambda_1(A(\Gamma))<\lambda_1(A(\Gamma^{\prime\prime}))$ by the proof of Claim \ref{claim1.1} of Lemma \ref{77}, where $H_2$ is $K_{2,2}$-minor free graph, $H_2\ncong U_1$ and contains only one $3$-length cycle. This also contradicts the maximality of $\lambda_1(A(\Gamma))$. If $H \cong U_2$, then $\lambda_1(A((K_n,D_{1,n-3}^-)))>\lambda_1(A((K_n,U_2^-)))$ by Lemma \ref{Lem3}, a contradiction.

Thus, $H$ must be a tree and $H\cong D_{1,n-3}$ by Lemma \ref{t}.
\end{proof}

\begin{figure}[H]
	\centering
	\ifpdf
	\setlength{\unitlength}{1bp}%
	\begin{picture}(168.07, 89.08)(0,0)
		\put(0,0){\includegraphics{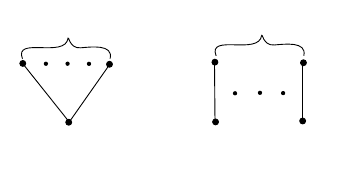}}
		\put(30.39,73.21){\fontsize{11.38}{13.66}\selectfont $s$}
		\put(123.57,74.52){\fontsize{11.38}{13.66}\selectfont $t$}
		\put(35.98,28.93){\fontsize{14.23}{17.07}\selectfont }
		\put(36.47,28.20){\fontsize{11.38}{13.66}\selectfont $v_1$}
		\put(5.67,51.00){\fontsize{11.38}{13.66}\selectfont $v_2$}
		\put(51.99,52.21){\fontsize{11.38}{13.66}\selectfont $v_{s+1}$}
		\put(81.75,54.15){\fontsize{11.38}{13.66}\selectfont $v_{s+2}$}
		\put(81.51,28.93){\fontsize{11.38}{13.66}\selectfont $v_{s+3}$}
		\put(147.28,55.12){\fontsize{11.38}{13.66}\selectfont $v_{n-1}$}
		\put(147.74,30.87){\fontsize{11.38}{13.66}\selectfont $v_n$}
		\put(25.21,1.12){\fontsize{11.38}{13.66}\selectfont $\bf{Fig.3.}$ The graph $K_{1,s}\cup tP_2$.}
	\end{picture}%
	\else
	\setlength{\unitlength}{1bp}%
	\begin{picture}(168.07, 89.08)(0,0)
		\put(0,0){\includegraphics{f4}}
		\put(30.39,73.21){\fontsize{11.38}{13.66}\selectfont s}
		\put(123.57,74.52){\fontsize{11.38}{13.66}\selectfont t}
		\put(35.98,28.93){\fontsize{14.23}{17.07}\selectfont }
		\put(36.47,28.20){\fontsize{11.38}{13.66}\selectfont 1}
		\put(5.67,51.00){\fontsize{11.38}{13.66}\selectfont 2}
		\put(51.99,52.21){\fontsize{11.38}{13.66}\selectfont s+1}
		\put(83.75,54.15){\fontsize{11.38}{13.66}\selectfont s+2}
		\put(83.51,28.93){\fontsize{11.38}{13.66}\selectfont s+3}
		\put(147.28,55.12){\fontsize{11.38}{13.66}\selectfont n-1}
		\put(148.74,30.87){\fontsize{11.38}{13.66}\selectfont n}
		\put(34.21,8.12){\fontsize{11.38}{13.66}\selectfont Fig.4. The graph .}
	\end{picture}%
	\fi
\end{figure}
\begin{lemma}\label{l10}
Let $s$ and $t$ be positive integers such that $2t+s+1=n$. Then $\lambda_1(A((K_n,D_{1,n-3}^-)))>\lambda_1(A((K_n,K_{1,n-3}^-\cup P_2^-)))\geq \lambda_1(A((K_n,K_{1,s}^-\cup tP_2^-)))$ with the second equality holds if and only if $s=n-3$, $t=1$.

\end{lemma}
\begin{proof}
 We give the $A((K_n,K_{1,s}^-\cup tP_2^-))$ and its corresponding quotient matrix $Q_5(s,t)$ by the vertex partition $V_1=\{v_1\}$, $V_2=\{v_2,...,v_{s+1}\}$, $V_3=\{v_{s+2},v_{s+3}\}$, $...$ , $V_{t+2}=\{v_{n-1},v_n\}$ as follows
 \begin{center}
$A((K_n,K_{1,s}^-\cup tP_2^-))=\begin{bmatrix}0&-j_s^T&j_{n-s-1}^T\\-j_s&\big(J-I\big)_s&J_{s\times(n-s-1)}\\j_{n-s-1}&J_{s\times(n-s-1)}^T&B_{n-s-1}\end{bmatrix}$ and $Q_{5(s,t)}=\begin{bmatrix}0&-s&2j_t^T\\-1&s-1&2j_t^T\\j_t&sj_t&\left(2J-3I\right)_t\end{bmatrix}$,
 \end{center}
where 
\begin{center}
$B_{n-s-1}=\begin{bmatrix}\left(I-J\right)_2&J_{2\times2}&\cdots&J_{2\times2}\\J_{2\times2}&\left(I-J\right)_2&&\vdots\\\vdots&&\ddots&J_{2\times2}\\J_{2\times2}&\cdots&J_{2\times2}&\left(I-J\right)_2\end{bmatrix}$.
\end{center}
Thus, 
\begin{center}
$\lambda I_{t+2}-Q_5(s,t)=\begin{bmatrix}\lambda&s&-2j_t^T\\1&\lambda+1-s&-2j_t^T\\-j_t&-sj_t&\left(\left(\lambda+3\right)I-2J\right)_t\end{bmatrix}$.
\end{center}
Now, we apply finitely many elementary row and column operations on the matrix $\lambda I_{t+2}-Q_5(s,t)$. First, subtracting the $3$-th row from all the lower rows and adding the $i$-th column to the $3$-th column for $i=t+2,...,4$. This leads to the following matrix:
\begin{center}
$\lambda I_{t+2}-Q_5(s,t)=\begin{bmatrix}\lambda&s&-2t&*\\1&\lambda+1-s&-2t&*\\-1&-s&\lambda+3-2t&*\\\bf{0}&\bf{0}&\bf{0}&(\lambda+3)I_{t-1}\end{bmatrix}$.
\end{center}
Note that the characteristic polynomial of $Q_5(s,t)$ is
\begin{center}
$P_{Q_5(s,t)}(\lambda)=(\lambda+3)^{t-1}(\lambda^3+(4-2t-s)\lambda^2+(3-4t-4s)\lambda+8st-3s-2t)$.
\end{center}
Adding $\alpha $J to the blocks of $A((K_n,K_{1,s}^-\cup tP_2^-))$, where $\alpha $ is constant. Then $A((K_n,K_{1,s}^-\cup tP_2^-))$ will be 
\begin{center}
$A_5=\begin{bmatrix}0&\bf{0}&\bf{0}\\\bf{0}&-I_s&\bf{0}\\\bf{0}&\bf{0}&I_{2t}\end{bmatrix}$.
\end{center}
Since $\lambda_1(Q_5(s,t))> 1$ and Spec$(A_5)$$=\begin{Bmatrix}-1^{[s]},0,1^{[2t]}\end{Bmatrix}$, $\lambda_1(A((K_n,K_{1,s}^-\cup tP_2^-)))=\lambda_1(Q_5(s,t))$. Let $f_{(s,t)}(\lambda)=(\lambda^3+(4-2t-s)\lambda^2+(3-4t-4s)\lambda+8st-3s-2t)$, then $f_{(s,t)}(\lambda)-f_{(s+2,t-1)}(\lambda)=4\lambda+8s-16t+20$ and $f_{(s-2,t+1)}(\lambda)-f_{(s,t)}(\lambda)=4\lambda+8s-16t-12$. Assume that $\lambda_1$ is the largest root of $f_{(s,t)}(\lambda)=0$. Next, we will discuss in two cases.
\begin{case4}
$4\lambda_1+8s-16t\geq0$.
\end{case4}
Then $f_{(s,t)}(\lambda_1)-f_{(s+2,t-1)}(\lambda_1)>0$, it means that $f_{(s+2,t-1)}(\lambda_1)<0$ and $\lambda_1(Q_5(s,t))<\lambda_1(Q_5(s$ $+2,t-1))$. Thus, $\lambda_1(A((K_n,K_{1,s}^-\cup tP_2^-)))<\lambda_1(A((K_n,K_{1,s+2}^-\cup (t-1)P_2^-)))$.
By repeatedly using this operation, we can obtain that
\begin{center}
$\lambda_1(A((K_n,K_{1,s}^-\cup tP_2^-)))\leq \lambda_1(A((K_n,K_{1,n-3}^-\cup P_2^-)))$,
\end{center}
the equality holds if and only if $s=n-3$, $t=1$.
\begin{case4}\label{44}
$4\lambda_1+8s-16t<0$.
\end{case4}

Then $f_{(s-2,t+1)}(\lambda_1)-f_{(s,t)}(\lambda_1)<0$, it implies that $f_{(s-2,t+1)}(\lambda_1)<0$ and $\lambda_1(Q_5(s,t))<\lambda_1(Q_5$ $(s-2,t+1))$. Hence, $\lambda_1(A((K_n,K_{1,s}^-\cup tP_2^-)))<\lambda_1(A((K_n,K_{1,s-2}^-\cup (t+1)P_2^-)))$. By repeatedly using this operation, we have $\lambda_1(A((K_n,K_{1,s}^-\cup tP_2^-)))\leq \lambda_1(A((K_n,P_3^-\cup \frac{n-3}2P_2^-)))$ if $n$ is odd, with the equality holds if and only if $s=2$, $t=\frac{n-3}2$. And $\lambda_1(A((K_n,K_{1,s}^-$ $\cup tP_2^-)))\leq \lambda_1(A((K_n,\frac n2 P_2^-)))$ if $n$ is even, with the equality holds if and only if $s=1$, $t=\frac{n}{2}-1$.

Note that $f_{(n-3,1)}(\lambda)=\lambda^3+(5-n)\lambda^2+(11-4n)\lambda+5n-17$, $f_{(2,\frac{n-3}{2})}(\lambda)=\lambda^{3}+(5-n)\lambda^{2}+(1-2n)\lambda+7n-27$ and $f_{(1,\frac{n}{2}-1)}(\lambda)=\lambda^{3}+(5-n)\lambda^{2}+(3-2n)\lambda+3n-9$. It is worth noting that
\begin{center}
 $f_{(1,\frac{n}{2}-1)}(\lambda)-f_{(n-3,1)}(\lambda)=(2n-8)(\lambda-1)$.
\end{center}
Note that $f_{(1,\frac{n}{2}-1)}(2)=25-5n<0$ for $n\geq 9$. Let $\lambda_1$ is the largest root of $f_{(1,\frac{n}{2}-1)}(\lambda)=0$. Then $\lambda_1>1$ and $f_{(n-3,1)}(\lambda_1)<0$ for $n \geq 9$. This means that $\lambda_1(Q_5(1,\frac{n}{2}-1))<\lambda_1(Q_5(n-3,1))$. Thus, $\lambda_1(A((K_n,\frac{n}{2} P_2^-)))<\lambda_1(A((K_n,K_{1,n-3}^-\cup P_2^-)))$. As well as, note that
\begin{center} $f_{(2,\frac{n-3}{2})}(\lambda)-f_{(n-3,1)}(\lambda)=(2n-10)(\lambda+1)$.
\end{center}
Set $\lambda_1$ is the largest root of $f_{(2,\frac{n-3}{2})}(\lambda)=0$. Then $f_{(n-3,1)}(\lambda_1)<0$ for $n \geq 9$. This means that $\lambda_1(Q_5(2,\frac{n-3}{2}))<\lambda_1(Q_5(n-3,1))$. Hence, $\lambda_1(A((K_n,P_3^-\cup \frac{n-3}2P_2^-)))<\lambda_1(A((K_n,K_{1,n-3}^-\cup P_2^-)))$.

Finally, we will show that $\lambda_1(A((K_n,K_{1,n-3}^-\cup P_2)))<\lambda_1(A((K_n,D_{1,n-3}^-)))$. 
By \cite{12}, the characteristic polynomial of $A((K_n,D_{1,n-3}^-))$ is $(\lambda+1)^{n-3}(\lambda^3+(3-n)\lambda^2+(3-2n)\lambda+7n-23)$, and set $g(\lambda)=\lambda^3+(3-n)\lambda^2+(3-2n)\lambda+7n-23$. Note that $f_{(n-3,1)}(\lambda)=\lambda^3+(5-n)\lambda^2+(11-4n)\lambda+5n-17$ and $f_{(n-3,1)}(n-3)=-2(n-4)^2<0$ for $n\geq 9$. Let $\lambda_1$ is the largest root of $f_{(n-3,1)}(\lambda)=0$, then $\lambda_1>n-3$. Set,
\begin{center}
$h(\lambda)=g(\lambda)-f_{(n-3,1)}(\lambda)=-2\lambda^2+(2n-8)\lambda+2n-6$,
\end{center}
and the maximal solution of $h(\lambda)=0$ is $n-3$ for $n \geq 9$. This means that $g(\lambda_1)<0$, then $\lambda_1(A((K_n,K_{1,n-3}^-\cup P_2)))<\lambda_1(A((K_n,D_{1,n-3}^-)))$.

So, the proof is completed.
\end{proof}
\begin{lemma}\label{l11}
Let $H$ be a disconnected $K_{2,2}$-minor free spanning subgraph of $K_n$ for $n \geq 9$. If $(K_n,H^-)$ has the maximum index, then $\lambda_1(A((K_n,D_{1,n-3}^-)))>\lambda_1(A((K_n,H^-)))$.
\end{lemma}
 \begin{proof}
 Let $\Gamma=(K_n,H^-)$ and $G_1,...,G_k$ be $k$ connected components of $H$ for $k \geq 2$. Cleary, $G_i$ is either a tree or a connected graph that only contains $3$-length cycles for $1 \leq i \leq k$. Let $X=(x_1,x_2,...,x_n)^T$ be a unit eigenvector associated with  $\lambda_1(A(\Gamma))$. Note that $-X$ must be a unit eigenvector of $\Gamma $ if $X$ is a unit eigenvector. Next, we will divide into the following two cases.
  \begin{case3}\label{C11}
There exists a nonnegative eigenvector.
 \end{case3}
 \noindent{\bf{${\mbox{subcase 1.1. }}$}} $x_i>0$ for any $1\leq i\leq n$. Firstly, we assert that $G_j$ is either a star graph or a $P_2$ for $1 \leq j\leq k$. Otherwise, we can find a $3$-length cycle or a $P_4$ as a subgraph of $H$. If $C_3=v_1v_2v_3v_1$ is a $3$-length cycle of $H$, then we can construct a new unbalanced signed complete graph $\Gamma^{\prime}$ whose negative edge-induced subgraph is still $K_{2,2}$-minor free by reversing the sign of the negative edge $v_{1}v_{2}$ such that $\lambda_1(A(\Gamma))<\lambda_1(A(\Gamma^\prime))$, a contradiction. If $P_4=v_4v_5v_6v_7$ is a subgraph of $H$, then we can construct a new unbalanced signed complete graph $\Gamma^{\prime}$ whose negative edge-induced subgraph is still $K_{2,2}$-minor free by reversing the sign of the negative edge $v_{5}v_{6}$ such that $\lambda_1(A(\Gamma))<\lambda_1(A(\Gamma^\prime))$, a contradiction. First, we consider that $H$ contains star subgraphs. Then we assert that $H$ contains only one star subgraph. Otherwise, without loss of generality, assume that  $F_1$ and $F_2$ are star subgraphs of $H$ with $v_1$and $v_3$ as their central vertices, respectively. Let $v_1v_2 \in E(F_1)$ and $v_3v_4 \in E(F_2)$. If $x_1 \geq x_3$, then the relocation $R(v_{2},v_{3},v_1)$ contradicts with the maximality of $\lambda_1(A(\Gamma))$ by Lemma \ref{o},  a contradiction. If $x_1 < x_3$, then the relocation $R(v_{4},v_{1},v_3)$ contradicts with the maximality of $\lambda_1(A(\Gamma))$ by Lemma \ref{o},  a contradiction. This means that $H \cong K_{1,s}^-\cup tP_2^-$, where $s\geq 2$ and $t\geq 1$. Thus, $\lambda_1(A((K_n,D_{1,n-3}^-)))>\lambda_1(A((K_n,K_{1,s}^-\cup tP_2^-)))$ by Lemma \ref{l10}. If $H$ does not contain star subgraphs. Then $H\cong \frac{n}{2}P_2$. Hence, $\lambda_1(A((K_n,D_{1,n-3}^-)))>\lambda_1(A((K_n,\frac{n}{2}P_2^-)))$ by the proof of Case \ref{44} of Lemma \ref{l10}.

  \noindent{\bf{${\mbox{subcase 1.2. }}$}} If there exists at least one vertex $v_i \in V(H)$ such that $x_i=0$. Without loss of generality, assume that $u \in V(G_1)$ such that $x_u=0$. Next, we assert that $k=2$. Otherwise, if there exists one vertex $v \in V(H)\backslash V(G_1)$ such that $x_v>0$, then we can construct a new unbalanced signed complete graph $\Gamma^{\prime}=(K_n,H_1^-)$ by reversing the sign of the positive edge $uv$ such that $\lambda_1(A(\Gamma))<\lambda_1(A(\Gamma^\prime))$, where $H_1$ is a disconnected $K_{2,2}$-minor free graph. This contradicts the maximality of $\lambda_1(A(\Gamma))$. If $x_v=0$ for any $v \in V(H)\backslash V(G_1)$, then there exists at least one vertex $w \in V(G_1)$ such that $x_w > 0$ since $X\neq \bf{0}$. Let $v\in V(H)\backslash V(G_1)$, then we can construct a new unbalanced signed complete graph $\Gamma^{\prime\prime}=(K_n,H_2^-)$ by reversing the sign of the positive edge $wv$ such that $\lambda_1(A(\Gamma))<\lambda_1(A(\Gamma^{\prime\prime}))$, where $H_2$ is a disconnected $K_{2,2}$-minor free graph. This also contradicts the maximality of $\lambda_1(A(\Gamma))$. Hence, $k=2$. That is, $H=G_1 \cup G_2$. If there exists one vertex $v \in V(G_2)$ such that $x_v>0$, then we can construct a new unbalanced signed complete graph $\Gamma_1=(K_n,H_1^-)$ by reversing the sign of the positive edge $uv$ such that $\lambda_1(A(\Gamma))<\lambda_1(A(\Gamma_1))$, where $H_1$ is a connected $K_{2,2}$-minor free graph and $H_1 \ncong U_1$. However, $\lambda_1(A((K_n,D_{1,n-3}^-)))\geq  \lambda_1(A(\Gamma_1))$ with the equality holds if and only if $H_1\cong D_{1,n-3}$ by Lemma \ref{l9}. Thus, $\lambda_1(A((K_n,D_{1,n-3}^-)))>\lambda_1(A((K_n,H^-)))$. If $x_i=0$ for any $v_i \in V(G_2)$, then there exists at least one vertex $w \in V(G_1)$ such that $x_w >0$ since $X\neq \bf{0}$. Let $v \in V(G_2)$, then we can construct a new unbalanced signed complete graph $\Gamma_2=(K_n,H_2^-)$ by reversing the sign of the positive edge $wv$ such that $\lambda_1(A(\Gamma))<\lambda_1(A(\Gamma_2))$, where $H_2$ is a connected $K_{2,2}$-minor free graph and $H_2 \ncong U_1$. However, $\lambda_1(A((K_n,D_{1,n-3}^-)))\geq  \lambda_1(A(\Gamma_2))$ with the equality holds if and only if $H_2\cong D_{1,n-3}$ by Lemma \ref{l9}. Thus, $\lambda_1(A((K_n,D_{1,n-3}^-)))>\lambda_1(A((K_n,H^-)))$.

 \begin{case3}
There are no nonnegative eigenvectors.
 \end{case3}
Without loss of generality, assume that $u$, $v\in V(H)$ such that $x_u\leq 0$, $x_v> 0$. Then we assert that $k=2$. Otherwise, we first assume that vertices $u$ and $v$ are in the same connected component, without loss of generality, we assume that $u$, $v\in V(G_1)$ and $w\in V(G_2)$. If $x_w>0$, then we can construct a new unbalanced signed complete graph $\Gamma^{\prime}=(K_n,H_1^-)$ by reversing the sign of the positive edge $uw$ such that $\lambda_1(A(\Gamma))<\lambda_1(A(\Gamma^\prime))$, where $H_1$ is a disconnected $K_{2,2}$-minor free graph. This contradicts the maximality of $\lambda_1(A(\Gamma))$. If $x_w\leq 0$, then we can construct a new unbalanced signed complete graph $\Gamma^{\prime\prime}=(K_n,H_2^-)$ by reversing the sign of the positive edge $vw$ such that $\lambda_1(A(\Gamma))<\lambda_1(A(\Gamma^{\prime\prime}))$, where $H_2$ is a disconnected $K_{2,2}$-minor free graph. This also contradicts the maximality of $\lambda_1(A(\Gamma))$. So, $u$ and $v$ are not in the same connected component, then we can construct a new unbalanced signed complete graph $\Gamma^{\prime\prime\prime}=(K_n,H_3^-)$ by reversing the sign of the positive edge $uv$ such that $\lambda_1(A(\Gamma))<\lambda_1(A(\Gamma^{\prime\prime\prime}))$, where $H_3$ is a disconnected $K_{2,2}$-minor free graph. This also contradicts the maximality of $\lambda_1(A(\Gamma))$. Thus, $k=2$. That is, $H=G_1 \cup G_2$. Let $u$, $v\in V(G_1)$ and $w\in V(G_2)$. If $x_w>0$, then we can construct a new unbalanced signed complete graph $\Gamma_1=(K_n,H_1^-)$ by reversing the sign of the positive edge $uw$ such that $\lambda_1(A(\Gamma))<\lambda_1(A(\Gamma_1))$, where $H_1$ is a connected $K_{2,2}$-minor free graph and $H_1 \ncong U_1$. However, $\lambda_1(A((K_n,D_{1,n-3}^-)))\geq  \lambda_1(A(\Gamma_1))$ with the equality holds if and only if $H_1\cong D_{1,n-3}$ by Lemma \ref{l9}. Thus, $\lambda_1(A((K_n,D_{1,n-3}^-)))>\lambda_1(A((K_n,H^-)))$. If $x_w\leq 0$, then we can construct a new unbalanced signed complete graph $\Gamma_2=(K_n,H_2^-)$ by reversing the sign of the positive edge $vw$ such that $\lambda_1(A(\Gamma))<\lambda_1(A(\Gamma_2))$, where $H_2$ is a connected $K_{2,2}$-minor free graph and $H_2 \ncong U_1$. However, $\lambda_1(A((K_n,D_{1,n-3}^-)))\geq  \lambda_1(A(\Gamma_2))$ with the equality holds if and only if $H_2\cong D_{1,n-3}$ by Lemma \ref{l9}. So, $\lambda_1(A((K_n,D_{1,n-3}^-)))>\lambda_1(A((K_n,H^-)))$. If $u\in V(G_1)$ and $v\in V(G_2)$, then we can construct a new unbalanced signed complete graph $\Gamma_3=(K_n,H_3^-)$ by reversing the sign of the positive edge $uv$ such that $\lambda_1(A(\Gamma))<\lambda_1(A(\Gamma_3))$, where $H_3$ is a connected $K_{2,2}$-minor free graph and $H_3 \ncong U_1$. However, $\lambda_1(A((K_n,D_{1,n-3}^-)))\geq  \lambda_1(A(\Gamma_3))$ with the equality holds if and only if $H_3\cong D_{1,n-3}$ by Lemma \ref{l9}. Hence, $\lambda_1(A((K_n,D_{1,n-3}^-)))>\lambda_1(A((K_n,H^-)))$.
 
So, the proof is completed.
\end{proof}
\noindent{\bf{Proof of Theorem \ref{thm1}}.} For $5\leq n\leq 8$, we can conclude that Theorem \ref{thm1} holds through direct calculation. For $n\geq 9$,
by Lemmas \ref{77}, \ref{l9} and \ref{l11}, the Theorem \ref{thm1} holds directly.\\\\
{\bf{Author Contributions}} All authors jointly worked on the results and they read and approved the final manuscript.\\\\
{\bf{Data Availability}} Not applicable. The manuscript has no associated data.
\section*{Acknowledgements}
The authors would like to show great gratitude to anonymous referees for their valuable suggestions which lead to an improvement of the original manuscript.
\section*{Declarations}

\noindent{\bf{Conflict of interest}} The authors declare that they have no conflict of interest.

\end{document}